\documentclass[11pt]{amsart}
\usepackage{graphicx}
\usepackage[mathcal]{euscript}
\usepackage{float}
\usepackage{amsmath,amsthm,amssymb,amsfonts,amscd,epsfig,latexsym,graphicx,textcomp}
\usepackage{tikz-cd}

\restylefloat{figure}

\newtheorem{Theorem}{Theorem}
\newtheorem*{theorem*}{Main Theorem}

\usepackage{thmtools}
\usepackage{thm-restate}

\usepackage{hyperref}

\usepackage{cleveref}

\newtheorem{question}[Theorem]{Question}

\newtheorem{theorem}{Theorem}[section]
\newtheorem{lemma}[theorem]{Lemma}
\newtheorem{corollary}[theorem]{Corollary}

\newtheorem{conjecture}[theorem]{Conjecture}
\newtheorem{scholium}[theorem]{Scholium}

\theoremstyle{definition}
\newtheorem{definition}[theorem]{Definition} 
\newtheorem{example}[theorem]{Example}

\theoremstyle{remark}
\newtheorem{remark}[theorem]{Remark}

\numberwithin{equation}{section}

\newcommand{\BR}{\mathbb{R}}

\newcommand{\ZZ}{{\mathbb Z}}
\newcommand{\RR}{{\mathbb R}}
\newcommand{\CC}{{\mathbb C}}

\begin{document}

\title{Lifting Lagrangian immersions in $\CC P^{n-1}$ to Lagrangian cones in $\CC^n$}

\author[S. Baldridge]{Scott Baldridge}
\author[B. McCarty]{Ben McCarty}
\author[D. Vela-Vick]{David Shea Vela-Vick}

\address{Department of Mathematics, Louisiana State University \newline
\hspace*{.375in} Baton Rouge, LA 70817, USA} \email{\rm{sbaldrid@math.lsu.edu}}

\address{Department of Mathematical Sciences, University of Memphis \newline
\hspace*{.375in} Memphis, TN, 38152, USA} \email{\rm{ben.mccarty@memphis.edu}}

\address{Department of Mathematics, Louisiana State University \newline
\hspace*{.375in} Baton Rouge, LA 70817, USA} \email{\rm{shea@math.lsu.edu}}

\subjclass{}
\date{}

\begin{abstract}
In this paper we show how to lift Lagrangian immersions in $\CC P^{n-1}$ to produce Lagrangian cones in $\CC^n$, and use this process to produce several families of examples of Lagrangian cones and special Lagrangian cones. Moreover we show how to produce Lagrangian cones, isotopic to the Harvey-Lawson and trivial cones, whose projections to $\CC P^{n-1}$ are immersions with few transverse double points.  
\end{abstract}

\maketitle

\bigskip
\section{Introduction}
\bigskip

This paper focuses on creating models for Lagrangian cones.  The motivation for this paper arises from the String Theory model in physics.  According to the theory, our universe consists of the standard Minkowski space-time, $\RR^4$, together with a complex Calabi-Yau 3-fold, $X$.  Based upon physical grounds, the SYZ-Conjecture of Strominger, Yau, and Zaslov (cf. \cite{SYZ}) expects that this Calabi-Yau can be viewed as a fibration by 3-tori with some singular fibers.  However, the singular fibers are not well-understood.  The standard approach is to model them locally as special Lagrangian cones $C\subset \CC^3$ (by cone, we mean a subset $C\subset \CC^3$ such that $r\cdot C = C$ for any real number $r>0$).  Such a cone can be characterized by its link, $C\bigcap S^5$, which is a Legendrian surface.  

\medskip

Special Lagrangian cones in $\CC^3$ are solutions to nonlinear, degree 2 and 3 partial differential equations.  Many of the papers on the subject up to now have approached their study from this perspective, often by using examples from algebraic geometry.  However, given that the cone can be characterized by the Legendrian link, this topic is very closely related to the study of knotted Legendrian submanifolds, which connects it with a great deal of work done in the area of contact topology.  In that area, much progress has been made, at least in part, due to the fact that that there are many nice topological and combinatorial representations of such submanifolds.  In dimension 3, where the problem of understanding Legendrian submanifolds amounts to classifying Legendrian knots up to isotopy, such diagrams are easy to come by.  For instance, grid diagrams can be used to obtain combinatorial representations of both front and Lagrangian projections of Legendrian knots (cf. \cite{Brunn}, \cite{Cromwell}, \cite{Ng}, \cite{legend}, \cite{Ng} and \cite{BaldMcCar2}).  In higher dimensions, there are fewer such constructions.   In \cite{rotation}, Ekholm, Etnyre, and Sullivan present front spinning as a way of constructing one class of knotted Legendrian tori, showing that the theory of Legendrian submanifolds of $\RR^{2n+1}$ is at least as rich in higher dimensions as it is in dimension $3$.  To accomplish this, they extend the definition of Legendrian contact homology to $\RR^{2n+1}$.  In \cite{BaldMcCar2}, it was shown that knotted Legendrian tori could be constructed from Lagrangian hypercube diagrams, and it was shown how to compute several invariants from such a diagram.  In \cite{Peter}, Lambert-Cole showed how to generalize that construction to produce a product operation on Legendrian submanifolds.

With the appropriate setup, it is possible to construct models of Legendrian surfaces in $S^5$ so that the resulting cone in $\CC^3$ is Lagrangian, and in some cases, special Lagrangian.  The Main Theorem of this paper describes precisely the conditions under which an immersion into $\CC P^{n-1}$ lifts to an embedded Legendrian submanifold of $S^{2n-1}$ and therefore gives rise to a Lagrangian cone.  

\begin{theorem*}
\label{thm:B}
Let $\Sigma$ be a closed, connected, smooth $(n-1)$-manifold, and $f:\Sigma \rightarrow \CC P^{n-1}$ a Lagrangian immersion with respect to the integral symplectic form $\frac{1}{\pi}\omega_{FS}$.  Let $\pi: S^{2n-1} \rightarrow \CC P^{n-1}$ be the principle Hopf $S^1$-bundle with connection $1$-form $\frac{i}{\pi} \alpha$ where $\alpha = i_0^* \left( \frac{1}{2} \sum_{i=1}^n x_i dy_i - y_i dx_i \right)$ for the identity map $i_0 : S^{2n-1} \rightarrow \CC^n$.  For each chart $\Psi_j: B_j \times S^1 \rightarrow S^{2n-1}$, there exists a $1$-form $\tau_j$ such that $\Psi_j^*(\alpha) = \frac{1}{2} (dt-\tau_j)$ where $\tau_j = - \sum_{\substack{i=1 \\ i \neq j}}^n (x_i dy_i - y_i dx_i)$.

If
\begin{enumerate}
\item $\Gamma \int_\gamma \tau = 0 \text{ mod }2\pi$ for all $[\gamma] \in H_1(\Sigma; \ZZ)$, and 
\item for all distinct points $x_1,...,x_k \in \Sigma$ such that $f(x_1) = f(x_j)$ for all $j\leq k$, and a choice of path $\gamma_j$ from $x_1$ to $x_j$ in $\Sigma$ for $2 \leq j \leq k$, the set $\left\{ \left( \Gamma \int_{f(\gamma_j)} \tau \right) \text{ mod } 2\pi \ | \ 2\leq j \leq k \right\}$ has $k-1$ distinct values, none of which are equal to $0$,

\end{enumerate}
then $f:\Sigma \rightarrow \CC P^{n-1}$ lifts to an embedding $\tilde{f}: \Sigma\rightarrow S^{2n-1}$ such that the image (the lift) $\tilde{\Sigma}$ is a Legendrian submanifold of $(S^{2n-1},\alpha)$.  In turn, the cone $c\tilde{\Sigma}$ is Lagrangian in $\CC^n$ with respect to the standard symplectic structure $\omega_0$.  
\end{theorem*}

\begin{remark}
The integral $\Gamma \int_\gamma$ refers to a \emph{lifiting integral} defined in Definition~\ref{def:liftingIntegral}.
\end{remark}

\begin{remark}
The second condition of the Main Theorem is stated for multiple points in general, but in most examples, we will only be working with double points or $S^1$-families of double points.
\end{remark}

While the Main Theorem is quite general, often it is possible, and indeed simpler, to work within a single chart of $\CC P^{n-1}$.  In what follows we first prove a special case of the Main Theorem in which we begin with an immersion into a single chart (cf. Theorem~\ref{thm:A}).  Several families of examples will be produced using this version of the theorem.

\medskip

The remainder of the paper is organized as follows.  In Section~\ref{sec:LagSpec} we introduce the special Lagrangian condition, and two examples that will be expounded on later.  In Section~\ref{sec:thmA}, we discuss the background information leading to the statement of useful simplification of the Main Theorem (cf. Theorem~\ref{thm:A}), and various examples we can construct using it.  In Section~\ref{sec:thmB}, we prove the Main Theorem, and give an example of a lift using it.  Section~\ref{sec:LegendrianResults} explores the implications of the Main Theorem for the study of Legendrian submanifolds of $S^{2n-1}$.  Finally, Section~\ref{sec:MinHam} introduces some questions regarding the study of Hamiltonian minimal submanifolds using the Theorems and examples in this paper.  

\section{Lagrangian and Special Lagrangian Cones}
\label{sec:LagSpec}
To construct a local model for special Lagrangian cones, we work in the symplectic manifold $(\CC^n, \omega, \Omega)$ where $\CC^n$ has complex coordinates $(z_1,...,z_n)$, $\omega = \frac{i}{2}(dz_1\wedge d\overline{z_1}+...+dz_n\wedge d\overline{z_n})$ is the standard K{\"a}hler form, and $\Omega = dz_1\wedge...\wedge dz_n$ is the holomorphic volume form (cf. \cite{Haskins}).  

\medskip

\begin{definition}
A cone $C\subset \CC^n$ is special Lagrangian if it is Lagrangian and $Im \Omega |_C \equiv 0$ or, equivalently, if $C$ is calibrated with respect to $Re\Omega$.
\end{definition}

\medskip

As a useful first step, we will focus first on the construction of Lagrangian cones.  Observe that the kernel of the 1-form 
$$\alpha = \frac{1}{2}\left(x_1 dy_1-y_1 dx_1 + ... + x_n dy_n - y_n dx_n\right)$$ 
where $z_j = x_j + i y_j$, restricted to the unit sphere, generates the standard contact structure for $S^{2n-1}$ and that $\alpha = \iota_R \omega$, where $R = 2\left(\sum_{i=1}^n x_i \frac{\partial}{\partial x_i} + y_i \frac{\partial}{\partial y_i}\right)$.  This means that, given a Legendrian submanifold $\Sigma \subset S^{2n-1}$, the associated cone $c\Sigma$, obtained by scaling $\Sigma$ by positive real numbers is automatically Lagrangian.  Moreover, any Lagrangian cone with vertex at the origin, must intersect $S^{2n-1}$ in a Legendrian surface. Hence, with respect to the standard contact structure on $S^{2n-1}$ and the standard symplectic form on $\CC^n$, a given submanifold of $S^{2n-1}\subset \CC^n$ is Legendrian if and only the associated cone in $\CC^n$ is Lagrangian.

\medskip

The following examples will be explored further in this paper.

\begin{example}
\label{ex:TrivIntro}
The trivial cone is simply a Lagrangian copy of $\RR^{n} \subset \CC^n$.  In particular, the following is well-known and easy to check:

\begin{theorem}
\label{thm:triv:Lag}
If $f: \RR^n \rightarrow \CC^{n}$ is given by $(x_1,...,x_n) \mapsto (x_1 \eta_1, ... , x_n \eta_n)$, where $\eta = (\eta_1,...,\eta_n)$ is a complex vector with $\eta_j \neq 0$ for all $j$, then the image of $f$ is Lagrangian with respect to the standard symplectic form $\omega$.
\end{theorem}

\begin{example}
\label{ex:trivSpec}
For some choices of $\eta$ the trivial cone is special Lagrangian.  For example, when $n=3$ a direct calculation shows that for $\eta = \left(a_1 + i b_1,a_2 + i b_2,a_3 + i b_3\right)$, if 
$$a_2 a_3 b_1 + a_1 a_3 b_2 + a_1 a_2 b_3 - b_1 b_2 b_3 = 0$$ 
then the map $f: \RR^3 \rightarrow \CC^3$ given by $(x_1,...,x_n) \mapsto (x_1 (a_1 + i b_1), x_2 (a_2 + i b_2), x_3 (a_3 + i b_3))$  is a special Lagrangian cone.  
\end{example}

\end{example}

\begin{example}
\label{ex:HLIntro}
Example III.3.A in \cite{HarveyLawson} introduced one of the first nontrivial families of examples of special Lagrangian cones, collectively known as \emph{the Harvey-Lawson cone.}  In particular, they proved that the cone on the $(n-1)$--tori defined by the following two sets is a special Lagrangian cone:
$$T^+ = \left\{ \left( e^{i \theta_1}, ... , e^{i \theta_n}  \right) \in \CC^n \ | \ \theta_1 + ... + \theta_n = 0 \right\},$$
$$T^- = \left\{ \left( e^{i \theta_1}, ... , e^{i \theta_n}  \right) \in \CC^n \ | \ \theta_1 + ... + \theta_n = \pi \right\}.$$
Of course, observe that $T^- = - T^+$, and we may re-write $T^+$ as:
\begin{equation}
\label{eqn:HLCone}
T^+ = \left\{ \left( e^{i \theta_1}, ... , e^{i \theta_{n-1}}, e^{- i (\theta_1 + ... + \theta_{n-1})}  \right) \ | \ \theta_1,...,\theta_{n-1} \in S^1 \right\},
\end{equation}
and we will call the cone on $T^+$ \emph{the Harvey-Lawson cone.}
\end{example}

\section{Theorem~\ref{thm:A}}
\label{sec:thmA}
In this section we develop the main theorem that we use for constructing examples of embedded Legendrian submanifolds of $S^{2n-1}$ as lifts of Lagrangian immersions in $\CC P^{n-1}$. 

The local theory for lifting Lagrangian immersions into a symplectic manifold to some $S^1$-bundle over that manifold comes out of the theory of fiber bundles.  Given a $2n$-dimensional symplectic manifold $(X^{2n},\omega)$ with an integral symplectic form.  Let $\pi: L\rightarrow X^n$ be the complex line bundle such that $c_1(L)=[\omega]$.  By the theory of line bundles (cf. \cite{GriffithsHarris}), we know that there is a $1$-form $\eta$ on the unit circle bundle $P = U(L)$ such that $d\eta = \pi^*(\omega)$.  In this case, $i\eta \in \Omega^1(P;\RR)$ is called the connection $1$-form.  If $f:\Sigma^n \rightarrow X^{2n}$ is a Lagrangian immersion of a connected $n$-dimensional manifold $\Sigma$, then $[f(\Sigma^n)] \cap [\omega] = 0$ and the pull-back of the $S^1$-bundle $P$ over $\Sigma$ is trivial.  Given
\begin{center}
\begin{tikzcd}
f^*(P)  \arrow[r, "F"]  \arrow[d]	&  P  \arrow[d, "\pi"] & \\ 
\Sigma  \arrow[r, "f"]  		& X^{2n} \\
\end{tikzcd}
\end{center}
then $f^*(P) \cong \Sigma \times S^1$.  In turn, there exists a section $\sigma: \Sigma \rightarrow f^*(P)$ which gives an immersed submanifold $F(\sigma(\Sigma))$ of $P$ (cf. \cite{Wolfson}).  

It is easy to see that $\eta$ (along with positive multiples of $\eta$) is a contact form for $P$.  In general, $F(\sigma(\Sigma))$ will not be Legendrian with respect to $\eta$.  However, we can always use $\eta$ to lift a neighborhood $U$ of $x_0 \in \Sigma$ to a Legendrian submanifold of $P$.  

Using the diffeomorphism $f^*(P) \cong \Sigma \times S^1$ along with the section $\sigma(x) = (x,1)$, we can define a trivialization of $P|_U$ by $(x,e^{it})$ for $x\in U$ and $t\in \RR$.  For $x\in U$, let $\gamma$ be a path in $U$ from $\gamma(0) = x_0$ to $\gamma(1) = x_1$.  This path gives rise to a path $\Gamma$ in $P|_U$ using the holonomy of the connection $1$-form $F^*(\eta)$.  That is $\Gamma$ is the unique path such that $\Gamma(0) = (x_0,1)$, $\pi(\Gamma(s)) = \gamma(s)$, and $F^*(\eta)(\Gamma'(s)) = 0$ $\forall s\in (0,1)$.  Define the lift $\tilde{f}:U\rightarrow P$ by $\tilde{f}(x) = F(\Gamma(1))$.  

This map is independent of the path chosen in the contractible neighborhood U because $f$ is a Lagrangian immersion (the restricted holonomy group at $x_0$ is trivial). 

We can write this holonomy map down explicitly in terms of $\Sigma \times S^1$ and the section $\sigma$ given by coordinates $(x,e^{it})$ where $x\in \Sigma$ and $t\in \RR$.  Suppose 
$$F^*(\eta) = k (dt - \tau)$$
where $k\in \RR$ is a constant, and $\tau \in \Omega^1(\Sigma)$.  The solution $\Gamma$ is equivalent to a path $(\gamma(x), e^{it(x)})\in \Sigma\times S^1$ where 
$$t(x) = \int_\gamma \tau$$
is obtained by integrating $dt-\tau$ along $\gamma$, setting the result equal to $0$, and choosing $t(0) = 0$.

This solution defines a local Legendrian lift, $\tilde{f}$ of $U$ into $P$.  We get a global lift if, 
$$\int_\gamma \tau \in 2\pi \ZZ, \; \; \forall [\gamma]\in H_1(\Sigma).$$
In this case $f:\Sigma \rightarrow X$ lifts to a Legendrian immersion $\tilde{f}:\Sigma \rightarrow P$ (i.e. the local lift extends to all of $\Sigma$).  

If integrating $\tau$ along any path joining a pair of double points results in a non-zero answer (mod $2\pi$), then the lift $\tilde{f}$ is an embedding.  We summarize the discussion above as follows:

\begin{theorem}
\label{thm:GenLift}
Let $\Sigma^n$ be a connected $n$-manifold, $X^{2n}$ be a $2n$-dimensional symplectic manifold with integral symplectic form $\omega$, and $f:\Sigma \rightarrow X$ be a Lagrangian immersion.  Let $\pi: P\rightarrow X$ be the principle $S^1$-bundle with connection $1$-form $i\eta$ determined by $d\eta = \pi^*(\omega)$.  Suppose the section $\sigma: \Sigma \rightarrow f^*(P)$ defines coordinates $(x,e^{it})$ of the trivial bundle $F: f^*(P)\rightarrow P$ such that $F^*(\eta) = k(dt - \tau)$ where $k\in \RR$ is a constant and $\tau \in \Omega^1(\Sigma)$.  If
\begin{enumerate}
\item $\int_\gamma \tau \in 2\pi \ZZ \: \: \forall [\gamma]\in H_1(\Sigma; \ZZ)$, and
\item for all points $x_0, x_1\in \Sigma$ such that $f(x_0) = f(x_1)$ and any path $\gamma$ from $x_0$ to $x_1$ in $\Sigma$, $\int_\gamma \tau \neq 0 \text{ mod }2\pi$,
\end{enumerate}
then $f:\Sigma\rightarrow X$ lifts to $\tilde{f}:\Sigma \rightarrow P$ and the image (the lift) $\tilde{\Sigma}$ is a Legendrian submanifold of $P$.

\end{theorem}

Theorem~\ref{thm:GenLift} is beautifully general in that it describes exactly when immersions can be lifted, but it is far from helpful in describing how to construct such lifts by hand (or with the help of a computer).  For example, given a symplectic manifold $X$, like $\CC P^n$ (or $T^n$, $E(n)$, $Sym^n(\Sigma_g)$, etc), what chart system should we use to make the calculation easiest?  (Note that the standard chart system $U_i = \{[z_1 : ...: 1: ... : z_{n}] | z_i \in \CC\} \subset \CC P^{n-1}$ is not convenient for constructing lifts.)

Can a chart system of $X$ be chosen in such a way that the symplectic form $\omega$ is standard in each chart?  Can a chart system be chosen so that the principal $S^1$-bundle trivializes over each chart in such a way that $\eta$ has a nice (simple) form in each trivialization, and there is an obvious choice of sections so that $\tau$ also has a nice representation?  None of these questions are answered by Theorem~\ref{thm:GenLift} (because they are specific to $X$), but all of them are extremely important to being able to generate explicit examples of lifts that satisfy the restrictive requirements needed to be able to compute invariants like the Legendrian contact homology of the lifts.  

For these reasons, the following theorem is far more useful to us in computing the invariants of Lagrangian cones in $\CC^n$.  

\begin{theorem}
\label{thm:A}
Let $\Sigma$ be a closed, connected, smooth $(n-1)$-manifold, and $f : \Sigma \rightarrow B^{n-1}$ be a Lagrangian immersion with respect to the standard symplectic form $\omega_0$ of $\CC^{n-1}$.  Let $\tau = - \sum_{i=1}^{n-1} \left( x_i dy_i - y_i dx_i \right)$ be a $1$-form on $B^{n-1}$.  If 

\begin{enumerate}
\item $\int_{f(\gamma)} \tau \in 2\pi \ZZ,   \forall \gamma \in H_1(\Sigma; \ZZ), \text{ and }$

\item for all distinct points $x_1,...,x_k \in \Sigma$ such that $f(x_1) = f(x_j)$ for all $j\leq k$, and a choice of path $\gamma_j$ from $x_1$ to $x_j$ in $\Sigma$ for $2 \leq j \leq k$, the set $\left\{ \left( \int_{f(\gamma_j)} \tau \right) \text{ mod } 2\pi \ | \ 2\leq j \leq k \right\}$ has $k-1$ distinct values, none of which are equal to $0$,

\end{enumerate}
 then $\Sigma$ lifts to an embedded Legendrian submanifold $\tilde{\Sigma} \subset S^{2n-1}$ whose associated cone $c\tilde{\Sigma}$ is Lagrangian in $\CC^n$.
 
 \medskip
 
The lift, $\tilde{f} : \Sigma \rightarrow S^{2n-1} \subset \CC^n$, is given by $\tilde{f}(x) = e^{i t(x)} (f_1(x), ..., f_{n-1}(x), \sqrt{1-|f(x)|^2})$ where 
$$t(x) = \int_{f(\gamma)} \tau$$ 
for some path $\gamma$ from an initial point $x_0 \in \Sigma$ to $x$.  

\end{theorem}

A careful comparison of the calculations in Theorem~\ref{thm:A} with those of Theorem~\ref{thm:GenLift} shows that Theorem~\ref{thm:A} is the realization of Theorem~\ref{thm:GenLift} in the case where $\Sigma^{n-1}$ is an immersion into an open unit ball, thought of as a single chart of $\CC P^{n-1}$ (and where we do the calculations in the chart, instead of in $\Sigma$).  For a proof of Theorem~\ref{thm:A}, see Section~\ref{sec:thmB}, where we prove the Main Theorem, which is a more general version of this theorem.

\section{Examples of lifts using Theorem~\ref{thm:A}}

\subsection{The Harvey-Lawson Special Lagrangian Cone}

\begin{example}
\label{ex:HL}
Theorem~\ref{thm:A} allows us to construct a family of isotopies of the famous example given by Harvey and Lawson (cf. Example III.3.A in \cite{HarveyLawson}).   Choose $\epsilon$ so that $0 \leq \epsilon < \sqrt{\frac{2}{n}}$ and define $\delta = \sqrt{ \frac{1}{n} - \frac{\epsilon^2}{2}}$.  Parametrize the torus, $T^{n-1}$, in the usual way with coordinates $(\theta_1,...,\theta_{n-1})\in \RR^{n-1}$.  Let $r_\epsilon (\theta_1,...,\theta_{n-1}) = \delta + \epsilon \sin(\theta_1 +...+ \theta_{n-1})$, and define $f_\epsilon : T^{n-1} \rightarrow B^{n-1}$ by:

$$f_\epsilon(\theta_1,...,\theta_{n-1}) = \left(r_\epsilon(\theta_1,...,\theta_{n-1}) e^{i(2 \theta_1 +\theta_2+...+ \theta_{n-1})}, ... , r_\epsilon(\theta_1,...,\theta_{n-1}) e^{i (\theta_1 + ... + \theta_{n-2}+2 \theta_{n-1})}\right).$$ 

Observe that the first condition of Theorem~\ref{thm:A} is satisfied.  Thus, defining $t(x)$ as in Theorem~\ref{thm:A}, we obtain a family of Legendrian tori in $S^{2n-1} \subset \CC^n$, each of whose associated cones are Lagrangian, given by the following maps:
\begin{multline*}
\tilde{f_\epsilon} (\theta_1,...,\theta_{n-1}) = \\
e^{i t_\epsilon (\theta_1,...,\theta_{n-1})} \left(r_\epsilon(\theta_1,...,\theta_{n-1}) e^{i(2 \theta_1 +\theta_2+...+ \theta_{n-1})}, ... , r_\epsilon(\theta_1,...,\theta_{n-1}) e^{i (\theta_1 + ... +\theta_{n-2} +  2 \theta_{n-1})}, \sqrt{1-(n-1)r_\epsilon^2} \right),
\end{multline*} 
where 
$$t_\epsilon (\theta_1,...,\theta_{n-1}) = \int_{f_\epsilon(\gamma)} \tau,$$ 
as in Theorem~\ref{thm:A}.  

\begin{theorem}
\label{thm:tCalc}
The parameter $t_\epsilon$ is given by the following:
$$t_\epsilon(\theta_1,...,\theta_{n-1}) = - (\theta_1 + ... + \theta_{n-1} ) - 2 \delta \epsilon \cos(\theta_1 + ... + \theta_{n-1}) + \frac{1}{4} \epsilon^2 \sin (2(\theta_1 + ... + \theta_{n-1})). $$
\end{theorem}

\begin{proof}
For simplicity, we work in polar coordinates for the computation below.  Taking $\gamma_i$ to be a path from $(\theta_1,...,\theta_{i-1},0,...,0)$ to $(\theta_1,...,\theta_{i-1},\theta_i,0,...,0)$, and $\gamma$ to be the concatenation of these paths from $i=1,..., n$, then we may solve for $t_\epsilon$ as follows:
\begin{eqnarray*}
t_\epsilon(\theta_1,...,\theta_{n-1}) & = & -n \sum_{i=1}^{n-1} \int_0^{\theta_i} r_\epsilon (\theta_1,...,\theta_{i-1}, \alpha_i,0,...,0)^2 d\alpha_i \\
& = & -n \sum_{i=1}^{n-1} \Bigl[ \Bigl( \frac{1}{2} (2\delta^2 + \epsilon^2) ( \theta_1 + ... + \theta_{i-1} + \alpha_i ) +  2 \delta \epsilon \cos ( \theta_1 + ... + \theta_{i-1} + \alpha_i )\\
& &  - \frac{1}{4} \epsilon^2 \sin (2(\theta_1+...+\theta_{i-1} + \alpha_i))  \Bigr) \Big|_0^{\theta_1} \Bigr] 
\end{eqnarray*}
Observe that the sum above telescopes, and hence, we may write
\begin{eqnarray*}
t_\epsilon(\theta_1,...,\theta_{n-1}) & = & -n \Bigl( \frac{1}{2} (2\delta^2 + \epsilon^2) (\theta_1 + ... + \theta_{n-1}) + 2 \delta \epsilon \cos (\theta_1 + ... + \theta_{n-1})\\
 & & - \frac{1}{4} \epsilon^2 \sin (2 (\theta_1 + ... + \theta_{n-1} ) ) \Bigr)\\
 & = & - (\theta_1 + ... + \theta_{n-1} ) - 2 \delta \epsilon \cos(\theta_1 + ... + \theta_{n-1}) + \frac{1}{4} \epsilon^2 \sin (2(\theta_1 + ... + \theta_{n-1})) 
\end{eqnarray*}

\end{proof}

In light of Theorem~\ref{thm:tCalc}, the following Corollary is obvious:

\begin{corollary}
\label{cor:tCalc}
As $\epsilon \rightarrow 0$, $\delta \rightarrow \frac{1}{\sqrt{n}}$, $t_\epsilon(\theta_1, ... ,\theta_{n-1}) \rightarrow t_0(\theta_1, ... ,\theta_{n-1})=- \theta_1 - ... - \theta_{n-1}$, and
$$\tilde{f}_\epsilon(\theta_1, ... ,\theta_{n-1}) \rightarrow \tilde{f}_0(\theta_1, ... , \theta_{n-1})= \frac{1}{\sqrt{n}}\left(e^{i\theta_1},...,e^{i\theta_{n-1}}, e^{-i(\theta_1+...+\theta_{n-1})}\right).$$
\end{corollary}

\begin{remark}
The cone on the image of the lift $\tilde{f}_\epsilon$ is Lagrangian for all $\epsilon \geq 0 $, but is also special Lagrangian when $\epsilon = 0$.  In fact, when $\epsilon = 0$, we the associated cone is the Harvey-Lawson cone (cf. Example~\ref{ex:HLIntro}).
\end{remark}

\medskip

In order to verify that the second condition of Theorem~\ref{thm:A} is satisfied, and consequently that the lift is embedded, we will be interested in locating the double points of $f_\epsilon$.   

For simplicity, we assume $n=3$ in the following calculation. The following lemma specifies precisely when the arguments of the exponential maps in the definition of $f_\epsilon$ all agree, a necessary condition for a double point.

\begin{lemma}
\label{lem:HLDoublePoints}
If $f_\epsilon(\theta_1,\theta_2) = f_\epsilon (\gamma_1, \gamma_2)$ then $\theta_1 = \gamma_1$ and $\theta_2 = \gamma_2$, or $\theta_1-\gamma_1 = \theta_2 - \gamma_2 = \frac{2\pi}{3} (\text{mod } 2\pi)$ or $\theta_1-\gamma_1 = \theta_2 - \gamma_2 = \frac{4\pi}{3} (\text{mod } 2\pi)$.
\end{lemma}

\begin{proof}
If $f_\epsilon(\theta_1,\theta_2)=f_\epsilon(\gamma_1,\gamma_2)$ then since the arguments of the exponential maps differ by a multiple of $2\pi$, $(\theta_1,\theta_2)$ and $(\gamma_1,\gamma_2)$ must satisfy the following equations:
\begin{equation}
\label{eqn:1}
2\theta_1+\theta_2 =  2\gamma_1+\gamma_2 + n 2\pi,
\end{equation}
\begin{equation}
\label{eqn:2}
\theta_1+2\theta_2 =  \gamma_1+2\gamma_2 + m 2\pi,
\end{equation}
for some $m,n\in \ZZ$.

Solving equations \ref{eqn:1} and \ref{eqn:2}, we obtain the following:
\begin{equation}
\label{eqn:3}
\theta_1-\gamma_1=\frac{2n-m}{3} 2\pi,
\end{equation}
\begin{equation}
\label{eqn:4}
\theta_2-\gamma_2=\frac{2m-n}2\pi.
\end{equation}
Since the torus $T^2$ is parametrized by $(\theta_1,\theta_2)\in [0,2\pi) \times [0,2\pi)$, it must be that $\theta_i-\gamma_i < 2\pi$ for $i=1,2$, and hence $|\frac{2m-n}{3}| < 1$ and $|\frac{2n-m}{3}| < 1$.  

Since $n,m \in \ZZ$, we find that the possibilities for $(n,m)$ are $\pm(1,0)$, $\pm(0,1)$, $\pm(1,1)$ and $(0,0)$.  Plugging these into equations \ref{eqn:3} and \ref{eqn:4}, we find that either $\theta_1 = \gamma_1$ and $\theta_2 = \gamma_2$, or $\theta_1-\gamma_1 = \theta_2 - \gamma_2 = \frac{2\pi}{3} (\text{mod } 2\pi)$ or $\theta_1-\gamma_1 = \theta_2 - \gamma_2 = \frac{4\pi}{3} (\text{mod } 2\pi)$.  
\end{proof}

In the proof above we also showed, after taking limits, that:

\begin{scholium}
\label{sch:3cover}
The image of $\tilde{f}_0$ is a 3-fold cover of the image of $f_0$ via the projection given by the Hopf map.  
\end{scholium}

Lemma~\ref{lem:HLDoublePoints} specifies when the arguments of the exponential maps will agree, but for a double point, the radii, determined by $r_\epsilon$ must also agree.  In the following lemma, we calculate where this occurs.

\begin{lemma}
\label{lem:HLDoublePoints2}
If $f_\epsilon(\theta_1,\theta_2) = f_\epsilon (\gamma_1, \gamma_2)$ and either $\theta_1-\gamma_1 = \theta_2 - \gamma_2 = \frac{2\pi}{3} (\text{mod } 2\pi)$ or $\theta_1-\gamma_1 = \theta_2 - \gamma_2 = \frac{4\pi}{3} (\text{mod } 2\pi)$, then one of the following must be true:
\begin{itemize}
\item $\theta_1 + \theta_2 = \gamma_1 + \gamma_2$,
\item $\theta_1+\theta_2 = \frac{7\pi}{6}$ and  $ \gamma_1 + \gamma_2 = \frac{11\pi}{6}$,
\item $\theta_1+\theta_2 = \frac{5\pi}{6}$ and $ \gamma_1 + \gamma_2 = \frac{\pi}{6}$.
\end{itemize}
\end{lemma}

\begin{proof}
Since $f_\epsilon(\theta_1,\theta_2) = f_\epsilon (\gamma_1, \gamma_2)$, not only must the arguments of the exponential maps differ by a multiple of $2\pi$, but the radii in each complex factor must match, that is $r_\epsilon(\theta_1,\theta_2) = r_\epsilon(\gamma_1, \gamma_2)$.  Hence one of the following equations must hold:
\begin{equation}
\label{eqn:5}
\theta_1+\theta_2 = \gamma_1+\gamma_2
\end{equation}
\begin{equation}
\label{eqn:6}
\theta_1+\theta_2+\gamma_1+\gamma_2=\pi + 2\pi k 
\end{equation}

There are several cases.  If $\theta_1+\theta_2 = \gamma_1+\gamma_2$, then using \ref{eqn:3} and \ref{eqn:4}, one can show that $n = -m$ which can only happen if $n=m=0$.  Furthermore, if $\theta_1+\theta_2+\gamma_1+\gamma_2=\pi + k 2\pi$, combining this with Equations \ref{eqn:3} and \ref{eqn:4}, we may solve the system to obtain that $\theta_1+\theta_2 = \frac{7\pi}{6}$ and $ \gamma_1 + \gamma_2 = \frac{11\pi}{6}$ or $\theta_1+\theta_2 = \frac{5\pi}{6}$ and $ \gamma_1 + \gamma_2 = \frac{\pi}{6}$.
\end{proof}

\begin{remark}
Lemma~\ref{lem:HLDoublePoints} rules out the possibility of multiple points of $f_\epsilon$ of multiplicity greater than 3, and Lemma~\ref{lem:HLDoublePoints2} shows that for $\epsilon>0$ there are no triple points.  Hence, immersion $f_\epsilon$ has only double points when $\epsilon>0$.
\end{remark}

The families of double points identified in Lemma~\ref{lem:HLDoublePoints2} form copies of $S^1$, and will show up not only in this example, but in others as well.  Hence the following definition will be useful in some of the discussion that follows.

\begin{definition}
\label{def:DPCircle}
Let $f: \Sigma \rightarrow M$ be an immersion of a surface.  Suppose $C_1$ and $C_2$ are disjoint copies of $S^1$ in $\Sigma$ such that $f(C_1)= f(C_2)$ and $f |_{C_1 \bigcup C_2}$ is a 2-to-1 map. Suppose further that $A_1$ and $A_2$ are disjoint annular neighborhoods of $C_1$ and $C_2$ and that $f(A_1) \bigcap f(A_2) = f(C_1)= f(C_2)$.  If for any pair consisting of $x_1 \in C_1$ and $x_2 \in C_2$ such that $f(x_1) = f(x_2)$ we have that $df_{x_1}(T A_1) \neq df_{x_2}(TA_2)$, then we call the image of $C_1$ and $C_2$ a \emph{double point circle}.

\end{definition}

\begin{theorem}
\label{thm:HLDoublePoints3}
The double points of $f_\epsilon$, of the form $f_\epsilon (\theta_1,\theta_2) = f_\epsilon(\gamma_1,\gamma_2)$, consist of two double point circles such that $\theta_1-\gamma_1 = \theta_2 - \gamma_2 = \frac{2\pi}{3} (\text{mod } 2\pi)$ or $\theta_1-\gamma_1 = \theta_2 - \gamma_2 = \frac{4\pi}{3} (\text{mod } 2\pi)$ and one of the following hold:
\begin{enumerate}
\item $\theta_1+\theta_2 = \frac{7\pi}{6}$ and  $ \gamma_1 + \gamma_2 = \frac{11\pi}{6}$,
\item $\theta_1+\theta_2 = \frac{5\pi}{6}$ and $ \gamma_1 + \gamma_2 = \frac{\pi}{6}$.
\end{enumerate}
\end{theorem}

\begin{proof}
Lemmas~\ref{lem:HLDoublePoints} and \ref{lem:HLDoublePoints2} demonstrate that systems of this type yield double points.  All that remains is the observation that If $(\theta_1,\theta_2)$ and $(\gamma_1,\gamma_2)$ satisfy $\theta_1-\gamma_1 = \theta_2 - \gamma_2 = \frac{2\pi}{3} (\text{mod } 2\pi)$ or $\theta_1-\gamma_1 = \theta_2 - \gamma_2 = \frac{4\pi}{3} (\text{mod } 2\pi)$ but do not satisfy either (1) or (2), then $sin(\theta_1+\theta_2) \neq sin(\gamma_1+\gamma_2)$.  For such cases, $r_\epsilon(\theta_1,\theta_2)\neq r_\epsilon(\gamma_1,\gamma_2)$ and hence $f_\epsilon(\theta_1,\theta_2)\neq f_\epsilon(\gamma_1,\gamma_2).$
\end{proof}
 
\begin{theorem}
\label{thm:embeddedHL}
The lift, $\tilde{f}_\epsilon$, is an embedding.
\end{theorem}

\begin{proof}
We already know the lift is well-defined.  All that remains to check that the second condition of Theorem~\ref{thm:A} is satisfied, which means that the double points of the projection are separated in the lift.  This amounts to computing $\int_{f(\gamma)} \tau$ for some path $\gamma$ joining a pair of double points of a double point circle.  Using Theorem~\ref{thm:HLDoublePoints3}, suppose we have a double point such that $f_\epsilon (\theta_1, \frac{5 \pi}{6} - \theta_1) = f_\epsilon (\theta_1 + \frac{2\pi}{3}, \frac{13\pi}{6} - (\theta_1 + \frac{2\pi}{3}))$.  Then the integral in question is given by:
$$t_\epsilon \left(\theta_1 + \frac{2\pi}{3}, \frac{13\pi}{6} - \left(\theta_1 + \frac{2\pi}{3}\right)\right) - t_\epsilon \left(\theta_1, \frac{5 \pi}{6} - \theta_1\right)$$
Using the expression for $t_\epsilon$ given in Theorem~\ref{thm:tCalc}, and simplifying, we obtain:
$$t_\epsilon \left(\theta_1 + \frac{2\pi}{3}, \frac{13\pi}{6} - \left(\theta_1 + \frac{2\pi}{3}\right)\right) - t_\epsilon \left(\theta_1, \frac{5 \pi}{6} - \theta_1\right)  =  -\frac{8\pi}{6} + 4\delta \epsilon \cos\left(\frac{5\pi}{6}\right) + \frac{\epsilon^2}{2} \sin \left(\frac{\pi}{3}\right).$$
Noting that $\epsilon < \sqrt{\frac{2}{3}}$, and $\delta = \sqrt{\frac{1}{3} - \frac{\epsilon^2}{2}}$, we have that 
$$ -\frac{8\pi}{6} - \frac{1}{3} < t_\epsilon \left(\theta_1 + \frac{2\pi}{3}, \frac{13\pi}{6} - \left(\theta_1 + \frac{2\pi}{3}\right)\right) - t_\epsilon \left(\theta_1, \frac{5 \pi}{6} - \theta_1\right) < -\frac{8\pi}{6} + \frac{1}{3}.$$
\end{proof}

Let $L_\epsilon$ be the image of $f_\epsilon$ and let $\tilde{L}_\epsilon$ be the Legendrian torus given by the lift, $\tilde{f}_\epsilon$.  We wish to identify the generators of the $0$-filtration level of the Legendrian contact homology of $\tilde{L}_\epsilon$, which are determined by the double points of the Lagrangian projection.  Recall that in this case, the double points are actually double point circles, hence we need to perturb the map so that it is chord-generic.  We will demonstrate the perturbation for $n=3$, but the general solution is similar.  

\begin{lemma}
\label{lem:perturbation}
Let  $\tilde{f}_\epsilon : T^2 \rightarrow S^5$ be the Legendrian torus given by the map
$$\tilde{f}_\epsilon(\theta_1,\theta_2) = e^{i t_\epsilon(\theta_1,\theta_2)} \left(r_\epsilon(\theta_1,\theta_2) e^{i (2\theta_1+\theta_2)}, r_\epsilon(\theta_1,\theta_2)e^{i (\theta_1+2\theta_2)}, \sqrt{1-r_{1,\epsilon}^2 - r_{2,\epsilon}^2}  \right).$$
Choose a perturbation in the direction of the Reeb fiber, $s_\epsilon: T^2 \rightarrow S^1$, two perturbtations in the radial directions, $s_{i,\epsilon} : T^2 \rightarrow \RR$, for $i=1,2$, and define 
$$\tilde{g}_\epsilon (\theta_1,\theta_2) = e^{i(t_\epsilon(\theta_1,\theta_2) + s_\epsilon(\theta_1,\theta_2))} \left( r_{1,\epsilon}(\theta_1,\theta_2) e^{i(2\theta_1+\theta_2)},  r_{2,\epsilon}(\theta_1,\theta_2) e^{i(\theta_1 + 2\theta_2)}, \sqrt{1-r_{1,\epsilon}^2 - r_{2,\epsilon}^2} \right)$$
where $r_{i,\epsilon}(\theta_1,\theta_2) = r_\epsilon(\theta_1,\theta_2) + s_{i,\epsilon}(\theta_1,\theta_2)$ for $i = 1,2.$  If
\begin{enumerate}
\item $\frac{\partial s_\epsilon}{\partial \theta_1} + 2 r_\epsilon(\theta_1,\theta_2)(2 s_{1,\epsilon}(\theta_1,\theta_2) + s_{2,\epsilon}(\theta_1, \theta_2)) + 2 s_{1,\epsilon}(\theta_1,\theta_2)^2 + s_{2,\epsilon}(\theta_1,\theta_2)^2 = 0$ and
\item $\frac{\partial s_\epsilon}{\partial \theta_2} + 2 r_\epsilon(\theta_1,\theta_2)(s_{1,\epsilon}1(\theta_1,\theta_2) + 2 s_{2,\epsilon}(\theta_1, \theta_2)) + s_{1,\epsilon}(\theta_1,\theta_2)^2 + 2 s_{2,\epsilon}(\theta_1,\theta_2)^2 = 0$
\end{enumerate}
then the perturbation $\tilde{g}_\epsilon$ is a Legendrian torus having only transverse double points.

Moreover, for a given choice of $s_\epsilon$ the system is solved by 

$$s_{1,\epsilon}(\theta_1,\theta_2)  = -r_\epsilon(\theta_1,\theta_2) +\sigma \sqrt{r_\epsilon(\theta_1,\theta_2)^2 + \frac{1}{3} \left(\frac{\partial s_\epsilon}{\partial \theta_2} - 2 \frac{\partial s_\epsilon}{\partial \theta_1}\right)}$$ and 
$$s_{2,\epsilon}(\theta_1,\theta_2)  = -r_\epsilon(\theta_1,\theta_2) +\sigma \sqrt{r_\epsilon(\theta_1,\theta_2)^2 + \frac{1}{3} \left(\frac{\partial s_\epsilon}{\partial \theta_1} - 2 \frac{\partial s_\epsilon}{\partial \theta_2}\right)}$$
where $\sigma$ is $\pm 1$.  
\end{lemma}

\begin{proof}
The calculation is easiest if we work in polar coordinates and identify a neighborhood of the $\tilde{f}_\epsilon$ with $B_2 \times S^1$ (cf. the Main Theorem).  Note that we may write:
$$\tilde{f}_\epsilon(\theta_1,\theta_2) = \left(r_\epsilon(\theta_1,\theta_2),2\theta_1+\theta_2,r_\epsilon(\theta_1,\theta_2),\theta_1+2\theta_2, t_\epsilon(\theta_1,\theta_2)\right),$$
and we work with the perturbation in polar coordinates as well:
$$\tilde{g}_\epsilon (\theta_1,\theta_2) = \left( r_{1,\epsilon}(\theta_1,\theta_2),2\theta_1+\theta_2,  r_{2,\epsilon}(\theta_1,\theta_2), \theta_1 + 2\theta_2, t_\epsilon(\theta_1,\theta_2) + s_\epsilon(\theta_1,\theta_2)\right),$$

In these coordinates, we may identify the contact form $\alpha$ on $S^5$ with $\frac{1}{2}(dt - \tau)$ (for details of this calculation see the Main Theorem).  Pulling back $\alpha$ to $T^2$ via $\tilde{f}_\epsilon$ we obtain the form:
$$\tilde{f}_\epsilon^*(\alpha) = \left(\frac{\partial t_\epsilon}{\partial \theta_1} + 3 r_\epsilon(\theta_1,\theta_2)^2 \right) d\theta_1 + \left(\frac{\partial t_\epsilon}{\partial \theta_2} + 3 r_\epsilon(\theta_1,\theta_2)^2 \right) d\theta_2$$
Since $\tilde{f}_\epsilon$ is Legendrian, this is 0, and hence $\left(\frac{\partial t_\epsilon}{\partial \theta_1} + 3 r_\epsilon(\theta_1,\theta_2)^2 \right)=\left(\frac{\partial t_\epsilon}{\partial \theta_2} + 3 r_\epsilon(\theta_1,\theta_2)^2 \right) =0$.
Pulling back $\alpha$ using the perturbation $\tilde{g}_\epsilon$ we obtain:
\begin{multline*}
\tilde{g}_\epsilon^*(\alpha) = \left[ \left(\frac{\partial t_\epsilon}{\partial \theta_1} + 3 r_\epsilon(\theta_1,\theta_2)^2 \right) + \frac{\partial s_\epsilon}{\partial \theta_1} + 2 r_\epsilon (\theta_1,\theta_2) (2 s_{1,\epsilon} + s_{2,\epsilon}) + 2 s_{1,\epsilon}^2 +s_{2,\epsilon}^2 \right] d\theta_1 \\
+ \left[ \left(\frac{\partial t_\epsilon}{\partial \theta_2} + 3 r_\epsilon(\theta_1,\theta_2)^2 \right)\frac{\partial s_\epsilon}{\partial \theta_2} + 2 r_\epsilon (\theta_1,\theta_2) (s_{1,\epsilon} + 2 s_{2,\epsilon}) + s_{1,\epsilon}^2 +2 s_{2,\epsilon}^2  \right] d\theta_2
\end{multline*}
Noting that $\left(\frac{\partial t_\epsilon}{\partial \theta_1} + 3 r_\epsilon(\theta_1,\theta_2)^2 \right)=\left(\frac{\partial t_\epsilon}{\partial \theta_2} + 3 r_\epsilon(\theta_1,\theta_2)^2 \right) =0$, we have justified (1) and (2).  The last part is routine, and obtained by solving this system of equations, (1) and (2), for $s_{1,\epsilon}$ and $s_{2,\epsilon}$.
\end{proof}

\begin{theorem}
\label{thm:perturbation}
The map $g_\epsilon : T^2 \rightarrow B^2$,
$$g_\epsilon(\theta_1,\theta_{2}) = \left( r_{1,\epsilon}(\theta_1,\theta_2) e^{i(2\theta_1+\theta_{2})}, r_{2,\epsilon}(\theta_1,\theta_2) e^{i(\theta_1+2\theta_2)}  \right)$$
where $r_{1,\epsilon}(\theta_1,\theta_2) = \sqrt{r_\epsilon(\theta_1,\theta_2)^2 - \frac{2}{3}\epsilon \cos(\theta_1)}$ and  $r_{2,\epsilon}(\theta_1,\theta_2) = \sqrt{r_\epsilon(\theta_1,\theta_2)^2 + \frac{1}{3}\epsilon \cos(\theta_1)}$ is a perturbation of $f_\epsilon$ having exactly two transverse double points.  Moreover, the lift $\tilde{g}_\epsilon$
$$\tilde{g}_\epsilon (\theta_1,\theta_2) = e^{i(t_\epsilon(\theta_1,\theta_2) + s_\epsilon(\theta_1,\theta_2))} \left( r_{1,\epsilon}(\theta_1,\theta_2) e^{i(2\theta_1+\theta_2)},  r_{2,\epsilon}(\theta_1,\theta_2) e^{i(\theta_1 + 2\theta_2)}, \sqrt{1-r_{1,\epsilon}^2 - r_{2,\epsilon}^2} \right),$$
is Legendrian isotopic to $\tilde{f}_\epsilon$.
\end{theorem}

\begin{proof}
Choose $s_\epsilon (\theta_1,\theta_2) = \epsilon \sin(\theta_1)$.  Then the two maps $s_{1,\epsilon}$ and $s_{2,\epsilon}$ from Lemma~\ref{lem:perturbation} satisfy the following:
\begin{enumerate}
\item $r_{1,\epsilon}(\theta_1,\theta_2) = r_\epsilon(\theta_1,\theta_2) + s_{1,\epsilon}(\theta_1,\theta_2)  = \sqrt{r_\epsilon(\theta_1,\theta_2)^2 - \frac{2}{3}\epsilon \cos(\theta_1)}$
\item $r_{2,\epsilon}(\theta_1,\theta_2) = r_\epsilon(\theta_1,\theta_2) + s_{2,\epsilon}(\theta_1,\theta_2)  = \sqrt{r_\epsilon(\theta_1,\theta_2)^2 + \frac{1}{3}\epsilon \cos(\theta_1)}$
\end{enumerate}

\end{proof}

The following corollary is is obvious:

\begin{corollary}
\label{cor:perturbation}
Taking the limit as $\epsilon\rightarrow 0$, we have the following:
\begin{enumerate}
\item $t_\epsilon(\theta_1,\theta_2) \rightarrow t_0(\theta_1,\theta_2) = -\theta_1 -\theta_2$
\item $\tilde{g}_\epsilon(\theta_1,\theta_2) \rightarrow \tilde{g}_0(\theta_1,\theta_2) = \tilde{f}_0(\theta_1,\theta_2) = \frac{1}{\sqrt{2}}\left(e^{i\theta_1}, e^{i\theta_2}, e^{-i(\theta_1+\theta_2)} \right).$
\end{enumerate}
\end{corollary}

Corollary~\ref{cor:perturbation} shows that $\tilde{g}_0$ is the Harvey-Lawson cone (just as $\tilde{f}_0$ is).  What makes $\tilde{g}_\epsilon$ useful is that although it is isotopic to the Harvey-Lawson cone, it has much nicer double points.  In fact, it has only 4 transverse double points as observed in the following corollary.

\begin{corollary}
\label{cor:TransverseDP}
The double points of $g_\epsilon$ can be found directly, and we obtain 2 for each double point circle, for a total of 4 transverse double points:  
\begin{enumerate}
\item $g_\epsilon(\frac{2 \pi}{3},\frac{\pi}{6}) = g_\epsilon(\frac{4 \pi}{3},\frac{5 \pi}{6})$, 
\item $g_\epsilon(\frac{5 \pi}{3},\frac{7 \pi}{6}) = g_\epsilon(\frac{\pi}{3},\frac{11 \pi}{6})$, 
\item $g_\epsilon(\frac{2 \pi}{3},\frac{7 \pi}{6}) = g_\epsilon(\frac{4 \pi}{3},\frac{11 \pi}{6})$, and 
\item $g_\epsilon(\frac{5 \pi}{3},\frac{\pi}{6}) = g_\epsilon(\frac{\pi}{3},\frac{5 \pi}{6}).$
\end{enumerate}
\end{corollary}

\begin{proof}
Writing $g_\epsilon$ in polar coordinates, as in Lemma~\ref{lem:HLDoublePoints}, we see that any double points must be of the form $g_\epsilon(\theta_1, \theta_2) = g_\epsilon(\theta_1+ j \frac{2\pi}{3}, \theta_2 + j \frac{2\pi}{3})$ where $j$ is either $1$ or $2$, in order that the arguments of the exponential maps both differ by a multiple of $2\pi$.  Thus we get double points when we have the following two equations satisfied.
$$r_{1,\epsilon}(\theta_1,\theta_2) = r_{1,\epsilon}(\theta_1+ j \frac{2\pi}{3}, \theta_2 + j \frac{2\pi}{3}),$$
$$r_{2,\epsilon}(\theta_1,\theta_2) = r_{2,\epsilon}(\theta_1+ j \frac{2\pi}{3}, \theta_2 + j \frac{2\pi}{3}).$$
Solving this system of equations, we obtain the result.
\end{proof}

In summary, we have constructed a family of cones, each of which is isotopic to the Harvey-Lawson cone, but with the additional property that the projection to $\CC P^{n-1}$ has only 4 transverse double points, unlike the actual Harvey-Lawson cone which is a $3$-fold cover of its projection to $\CC P^{n-1}$, as observed in Scholium~\ref{sch:3cover}.  Although the isotopy taking the Harvey-Lawson cone to one of our perturbations does not preserve the special Lagrangian conditions, it does preserve the Legendrian link, and hence, preserves the Legendrian contact homology.  Moreover, our perturbations have only transverse double points. 

\end{example}

\subsection{Lagrangian hypercube diagrams}
In \cite{BaldMcCar2}, Lagrangian hypercube diagrams were used to produce examples of Legendrian tori in the standard contact space, $(\RR^5, \xi_{std})$, using $wxyz$-coordinates.  But they can also be adapted to produce Legendrian tori in $S^5$ whose cones in $\CC^3$ are Lagrangian.  Before doing so, we briefly recall some of the relevant material from \cite{BaldMcCar2} and refer the reader to that paper for more details.

Lagrangian hypercube diagrams are closely related to grid, cube, and hypercube diagrams.  To construct a grid, cube, or hypercube diagram, one places markings in a 2, 3, or 4 dimensional Cartesian grid, while ensuring that certain marking conditions and crossing conditions hold (cf.  Section 2 and 3 in \cite{Bald}, and Section 2 in \cite{ScottAdam}).  In each case, the markings determine a link (cf. Figure~\ref{fig:gridCubeHyper}).  For a hypercube diagram, there is an algorithm for constructing a Lagrangian torus associated to the hypercube diagram, such as the one shown in the last picture in Figure~\ref{fig:gridCubeHyper} (cf. Theorem 5.1 in \cite{Bald}).

\begin{figure}[h]
\includegraphics[scale = 1]{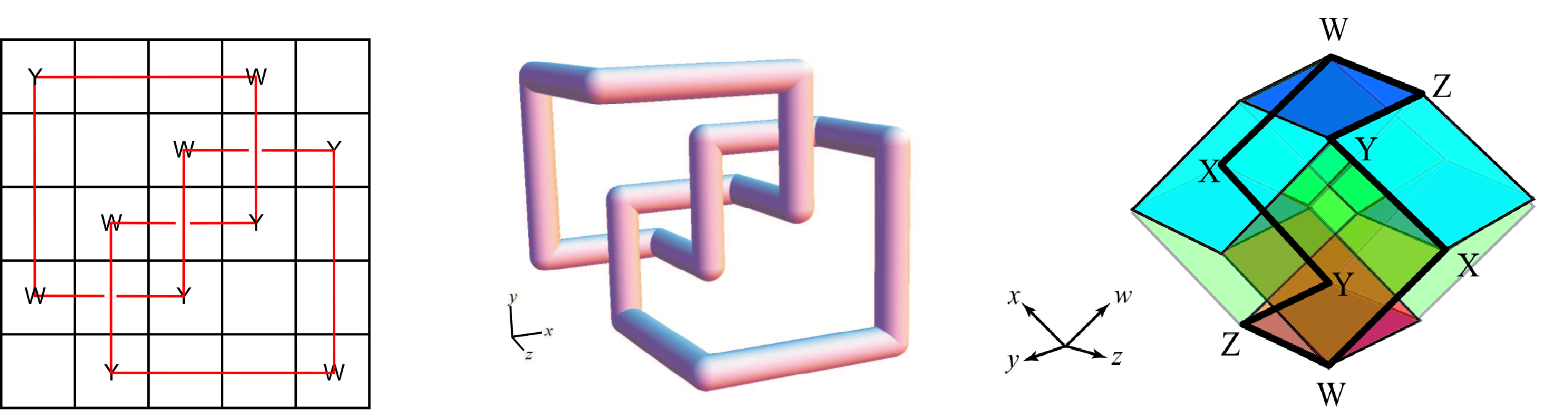}
\caption{Grid and cube diagrams for the trefoil, and a hypercube diagram for a torus.}
\label{fig:gridCubeHyper}
\end{figure}

In order to define a Lagrangian hypercube diagram, we first need to define a Lagrangian grid diagram:
\begin{definition}
A \emph{Lagrangian grid diagram} given by $\gamma: S^1 \rightarrow \RR^2$ where $\gamma(\theta) = \left(x(\theta),y(\theta)\right)$ is an immersed grid diagram $G$ satisfying Conditions~\ref{eqn:areaInt} and \ref{eqn:crossingInt}.
\begin{equation}
\label{eqn:areaInt}
\int_0^{2\pi} y(\theta) x'(\theta)d\theta = 0,
\end{equation}
\begin{equation}
\label{eqn:crossingInt}
\int_{\theta_0}^{\theta_1} y(\theta) x'(\theta)d\theta \neq 0 \text{ whenever } \gamma(\theta_0) = \gamma(\theta_1) \text{ and } 0 < \theta_1 - \theta_0 <2\pi. 
\end{equation}
\end{definition}

While any Lagrangian projection of a Legendrian knot satisfies Equation~\ref{eqn:areaInt} and \ref{eqn:crossingInt}, it is usually difficult to determine from a given diagram in the plane whether or not the diagram will lift to a Legendrian knot.  The advantage with a Lagrandian grid diagram is that one merely needs to add up the signed areas of a finite number of rectangles to determine whether the diagram lifts to a Legendrian knot (cf. Corollary 3.10, Scholium 3.12 and Corollary 3.13 in \cite{BaldMcCar2}).    

A Lagrangian hypercube diagram takes two Lagrangian grid diagrams and uses them to construct a \emph{product} of two Legendrian knots (cf \cite{Peter}, and \cite{BaldMcCar2}).  To construct a grid diagram, one places markings in a 2-dimensional grid, subject to a set of marking conditions, and creates a knot diagram by drawing segments, joining the markings to create immersed loops.  The process of creating Lagrangian hypercube diagram is similar: there is a set of marking conditions that determine how to place markings in a 4-dimensional Cartesian grid, and the markings are joined by segments, following an algorithm to create a simple loop. Before stating the conditions, we give a few preliminaries.

\medskip

 A \textit{flat} is any right rectangular $4$-dimensional polytope with integer valued vertices in $C$ such that there are two orthogonal edges at a vertex of length $n$ and the remaining two orthogonal edges are of length $1$. (Each flat is congruent to the product of a unit square and an $n\times n$ square.)  Moreover, the flat will be named by the two edges of length $n$.  Although a flat is a 4-dimensional object, the name references the fact that a flat is a 2-dimensional array of unit hypercubes.  For example, an $xy$-flat is a flat that has a face that is an $n\times n$ square that is parallel to the $xy$-plane.  In a hypercube of size $n=3$, one example of a $xy$-flat would be the subset $[0,1] \times [0,3] \times [0,3] \times [2,3]$ (shown in Figure~\ref{fig:cubesandflats}).

\medskip

A {\em stack} is a set of $n$ flats that form a right rectangular $4$-dimensional polytope with integer vertices in $C$ in which there are three orthogonal edges of length $n$ at a vertex, and the remaining edge has length $1$. (Each stack is the product of a cube with edges of length $n$ and a unit interval.)  A stack is named by the three edges of length $n$.  An example of a $wxz$-stack in a hypercube of size $3$ is the subset $[0,3] \times [0,3] \times [2,3]  \times [0,3]$ (shown at the top of Figure~\ref{fig:cubesandflats}).  Further examples of flats and stacks may be found in Figure~\ref{fig:cubesandflats}.

\medskip

A marking is a labeled point in $\BR^4$ with half-integer coordinates in $C$. Unit hypercubes of the $4$-dimensional Cartesian grid will either be blank, or marked with a $W$, $X$, $Y$, or $Z$ such that the following {\em marking conditions} hold:

\medskip

\begin{enumerate}
 \item each stack has exactly one $W$, one $X$, one $Y$, and one $Z$ marking;\\

 \item each stack has exactly two flats containing exactly 3 markings in each; \\

 \item for each flat containing exactly 3 markings, the markings in that flat form a right angle such that each ray is parallel to a coordinate axis;\\

    \item for each flat containing exactly 3 markings, the marking that is the vertex of the right angle is $W$ if and only if the flat is a $zw$-flat, $X$ if and only if the flat is a $wx$-flat, $Y$ if and only if the flat is a $xy$-flat, and $Z$ if and only if the flat is a $yz$-flat.

\end{enumerate}

\begin{figure}[h]
\includegraphics[scale=1]{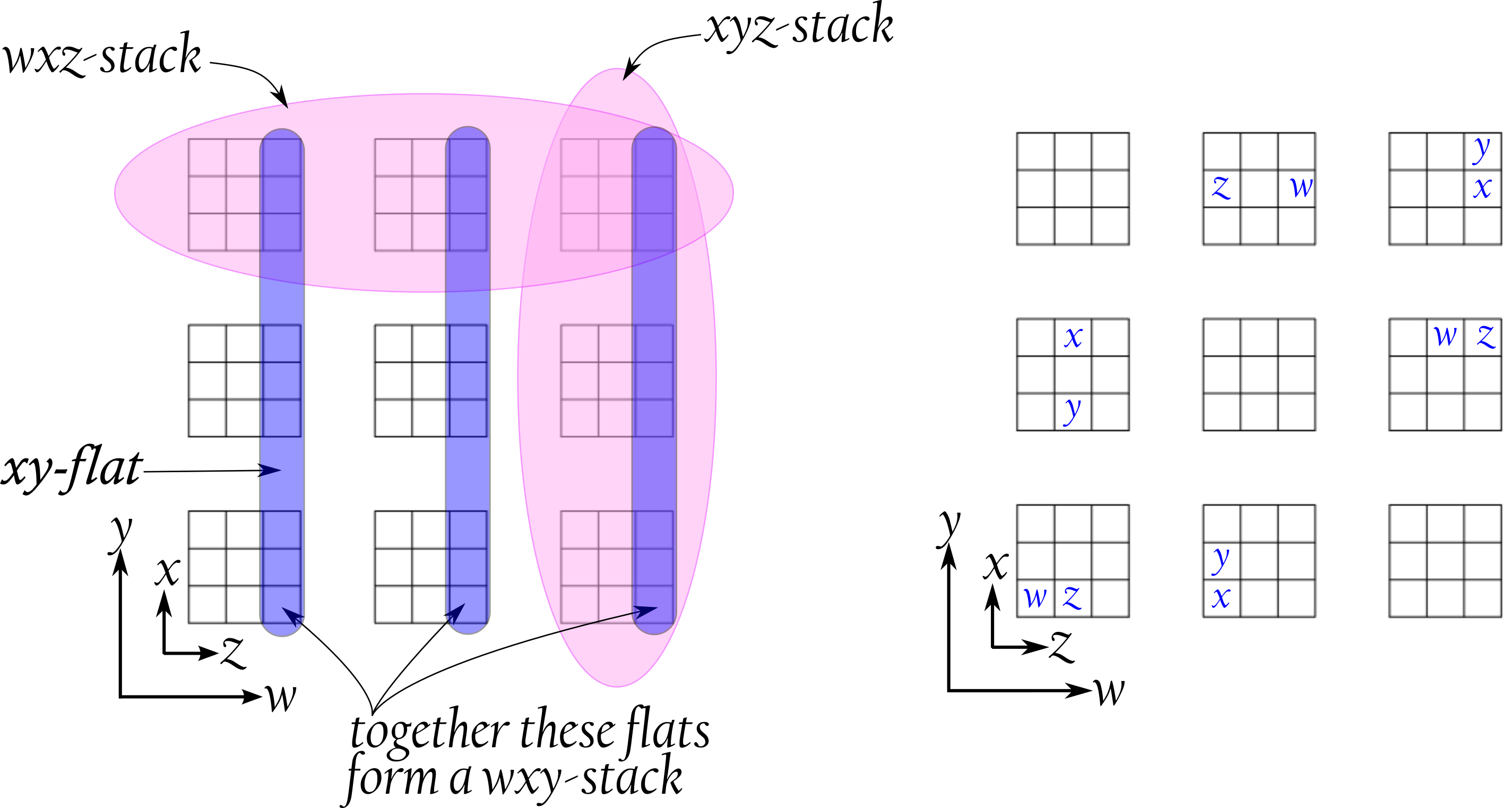}
\caption{\small \it  A schematic for displaying a Lagrangian hypercube diagram.  The outer $w$ and $y$ coordinates indicate the ``level'' of each $zx$-flat.  The inner $z$ and $x$ coordinates start at $(0,0)$ for each of the nine $zx$-flats. With these conventions understood, it is then easy to display $xy$-flats, $xyz$-stacks, $wxz$-stacks, $wxy$-stacks, etc.  The second picture is a schematic of a Lagrangian hypercube diagram.} \label{fig:cubesandflats}
\end{figure}

\medskip

Condition 4 rules out the possibility of either $wy$-flats or a $zx$-flats with three markings (see Figure~\ref{fig:cubesandflats}).  As with oriented grid diagrams and cube diagrams, we obtain an oriented link from the markings by connecting each $W$ marking to an $X$ marking by a segment parallel to the $w$-axis, each $X$ marking to a $Y$ marking by a segment parallel to the $x$-axis, and so on.  

\medskip

Let $\pi_{xz}, \pi_{wy} : \RR^4 \rightarrow \RR^2$ be the natural projections, projecting out the $x,z$ and $w,y$ directions respectively.  The projection $\pi_{xz}(C)$ produces an $n\times n$ square in the $wy$-plane.  If we project the $W$ and $Y$ markings of the hypercube to this square as well, the markings satisfy the conditions for an immersed grid diagram, which we denote $G_{wy} := (\pi_{xz}(C), \pi_{xz}(\mathcal{W}), \pi_{xz}(\mathcal{Y}))$, where $\mathcal{W}$ and $\mathcal{Y}$ are the sets of $W$ and $Y$ markings respectively.  Similarly, we define $G_{zx} := (\pi_{wy}(C), \pi_{wy}(\mathcal{Z}), \pi_{wy}(\mathcal{X}))$, where $\mathcal{Z}$ and $\mathcal{X}$ are the sets of $Z$ and $X$ markings respectively.  

\medskip

In a grid diagram, one typically requires a crossing condition, namely that the vertical segment crosses over the horizontal segment.  For a Lagrangian hypercube diagram, the crossing conditions are determined as follows.   We require that the two immersed grid diagrams, $G_{zx}$ and $G_{wy}$, are Lagrangian grid diagrams (that is, they satisfy Conditions~\ref{eqn:areaInt} and \ref{eqn:crossingInt}).  By Proposition 3.4 of \cite{BaldMcCar2}, a Lagrangian grid diagram lifts to a smoothly embedded Legendrian knot.  Hence the crossing conditions of the grid are determined by this lift.  We require one additional {\em product lift condition} that the pair $G_{zx}$ and $G_{wy}$ must satisfy ($\Delta t(c)$ in the definition below is the length of the Reeb chord associated to the crossing $c$).

\begin{definition}
\label{def:LagHyper}
For two Lagrangian grid diagrams, $G_{wy}$ and $G_{zx}$, let $\mathcal{C}=\{c_i\}$ be the crossings in $G_{zx}$ and $\mathcal{C'}=\{c_i'\}$ be the crossings in $G_{wy}$.  The pair of grid diagrams is said to satisfy the {\em product lift condition} if $|\Delta t(c_i)| \neq |\Delta t(c_i')|$ for all $i,j$.    
\end{definition}

\medskip

We are now ready to define a Lagrangian hypercube diagram (cf. \cite{BaldMcCar2}):

\begin{definition}
A {\em Lagrangian hypercube diagram}, denoted $H\Gamma = (C, \{\mathcal{W}, \mathcal{X},\mathcal{Y},\mathcal{Z}\}, G_{zx}, G_{wy})$, is a set of markings $\{\mathcal{W}, \mathcal{X},\mathcal{Y},\mathcal{Z}\}$ in $C$ that $(1)$ satisfy the marking conditions, $(2)$ $G_{wy}$ and $G_{zx}$ are Lagrangian grid diagrams, and $(3)$ $G_{wy}$ and $G_{zx}$ satisfy the product lift condition.  
\end{definition}

The immersed torus specified by the Lagrangian hypercube diagram is the product of $G_{zx}$ and $G_{wy}$, determined as follows:  place a copy of the immersed grid $G_{zx}$ at each $zx$-flat on the schematic that contains a pair of markings (shown in red on Figure~\ref{fig:unknot3}).  Doing so produces a schematic with two copies of $G_{zx}$ with the same $y$-coordinates and two with the same $w$-coordinates.  For each pair of copies sharing the same $w$-coordinates, we may translate one parallel to the $w$-axis toward the other.  Doing so traces out an immersed tube connecting these two copies of $G_{zx}$.  Similarly, we may translate parallel to the $y$-axis to produce an immersed tube connecting two copies of $G_{zx}$ with the same $y$-coordinates.  Since we are connecting copies of $G_{zx}$ in flats corresponding to the markings of $G_{wy}$, the tube will close to produce an immersed torus.  

\begin{figure}[h]
\includegraphics[scale=.6]{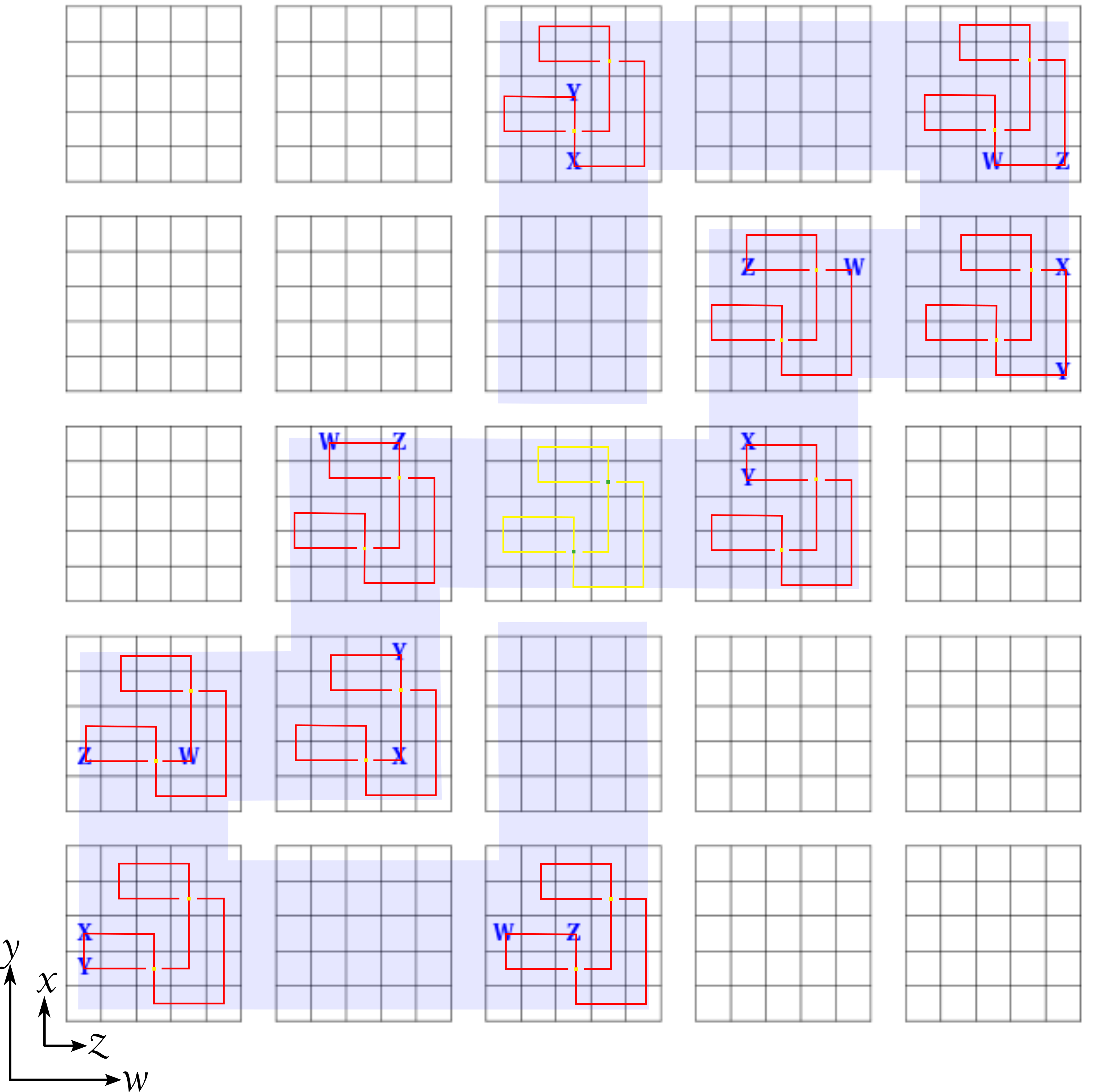}
\caption{Lagrangian hypercube diagram with unknotted $G_{zx}$ and $G_{wy}$ and rotation class $(1,0)$.}
\label{fig:unknot3}
\end{figure}

\subsection{Lagrangian cones in $\CC^3$ constructed from Lagrangian hypercube diagrams}
\label{sec:HypercubeCones}

First, we show how to convert a grid diagram to a \emph{radial grid diagram}.  A set of concentric circles $\{C_k\}_{k=1}^n$ of radius $\sqrt{\frac{k}{3n}}$ will serve to represent the rows of our grid, and a set of radial lines, determined by the list of angles, $\{k \frac{2 \pi}{n} \}_{k=0}^{n-1}$, to serve as columns.  The counterclockwise direction is chosen to correspond to the positive $x$-direction in the original grid, and the outward pointing radial direction is chosen to correspond to the positive $y$-direciton.  Moreover, the radii of the concentric circles are chosen so that each annular band has area $\frac{\pi}{3n}$ and consequently, each cell, as shown in Figure~\ref{fig:radialGrid}, has equal area (in particular, each cell has area $\frac{1}{n} \cdot \frac{\pi}{3n}$).

For a given marking in row $i$ and column $j$, we place it in the radial grid at the intersection of the circle, $C_i$, with the radial line segment determined by the angle $j \frac{2 \pi}{n}$ to obtain a radial grid diagram.  Join the markings in the radial grid diagram to match the original grid diagram (cf. Figure~\ref{fig:radialGrid}).

\begin{figure}[h]
\includegraphics[scale=.3]{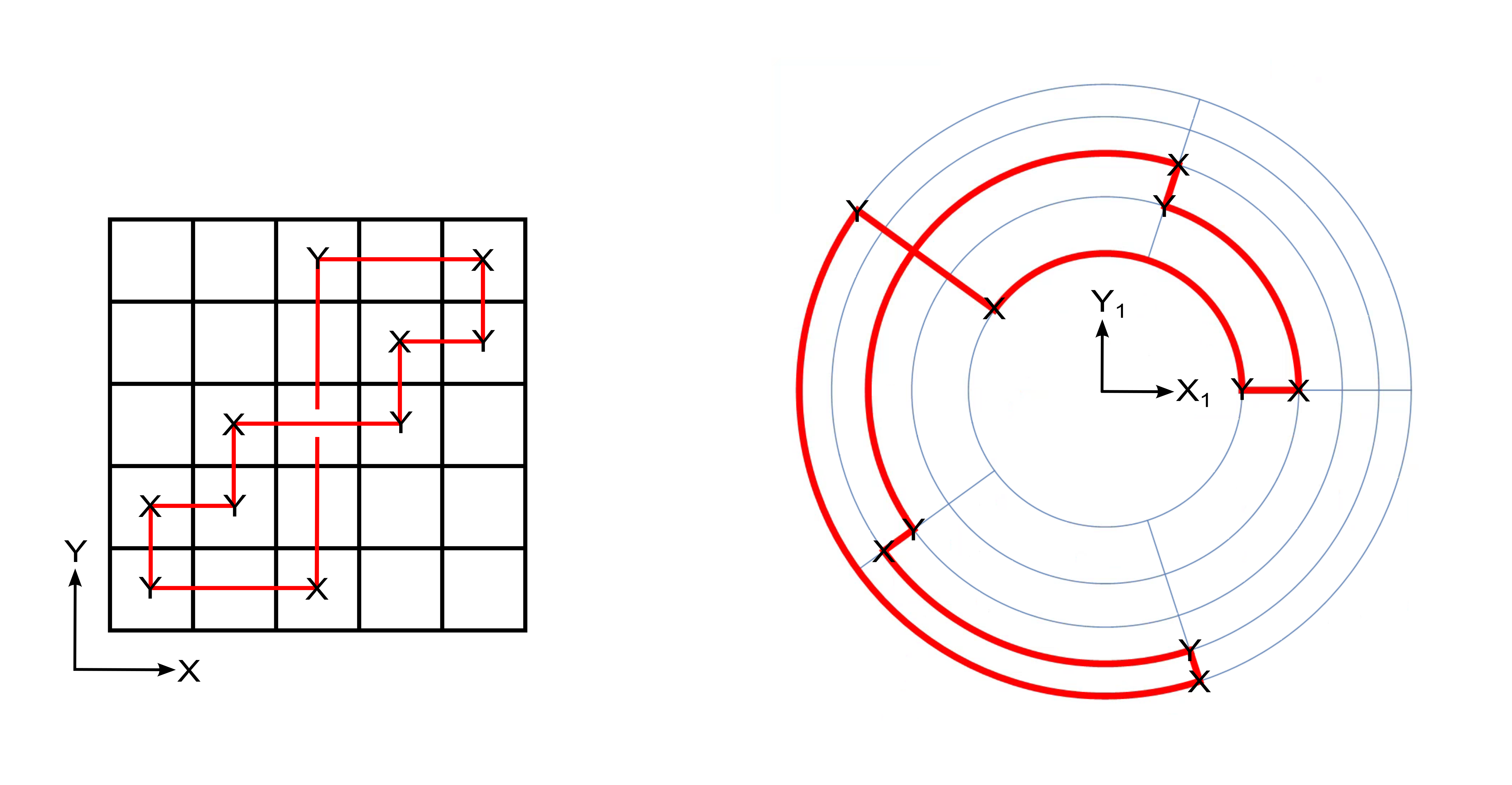}
\caption{Converting a $5\times5$ Lagrangian grid diagram to a radial Lagrangian grid diagram.}
\label{fig:radialGrid}
\end{figure}

\begin{remark} Notice that while the markings of the oriented grid diagram are placed in the cells of the grid, the markings of the radial grid diagram are placed at the intersections of the grid lines.  This is just a shift of the markings by $\left( -\frac{1}{2},-\frac{1}{2}\right)$.
\end{remark}

Suppose that $\hat{G}_{x_1y_1}$ and $\hat{G}_{x_2y_2}$ are radial grid diagrams constructed (as above) from Lagrangian grid diagrams $G_{x_1y_1}$ and $G_{x_2y_2}$.  We can define an immersion $f: T^2 \rightarrow B^2$ by letting $\gamma_1 : \theta_1 \mapsto (x_1(\theta_1),y_1(\theta_1))$ and $\gamma_2 : \theta_2 \mapsto (x_2(\theta_2),y_2(\theta_2))$ be the two loops corresponding to the radial grid diagrams $\hat{G}_{x_1y_1}$ and $\hat{G}_{x_2y_2}$.  

We wish to lift $f$ to a Legendrian torus in $S^5$ using Theorem~\ref{thm:A}, but to do so, it must first be smoothed.  This may be remedied by following a smoothing procedure as described in Theorem 3.9, Corollary 3.10, Scholium 3.12, and Corollary 3.13 of \cite{BaldMcCar2}, and noting that the integral used to define the lift in Theorem~\ref{thm:A} results in a net area calculation here, just as it was in \cite{BaldMcCar2}. To see this observe that for a path that follows a radial segment in one of the grids, the change in $t$ is 0.  For a path that follows a circular arc in one of the grids the contribution to the change in $t$ is given by $a r^2$ where $a$ is the subtended angle of the arc (positive if segment is oriented counterclockwise and negative otherwise), and $r$ is the radius of the arc.  That is to say, the magnitude of the change in $t$ along such an arc is twice the area of the sector it bounds (and positive if the arc run counterclockwise, and negative otherwise).  Since the radial grid is constructed so that every cell has equal area, the proofs of Theorem 3.9, Corollary 3.10, Scholium 3.12, and Corollary 3.13 in \cite{BaldMcCar2} may be easily adapted to this setting.  Combining this with Theorem~\ref{thm:A} we obtain the following:

\begin{theorem}
\label{thm:hypercubeLift}
Let $\hat{G}_{x_1y_1}$ and $\hat{G}_{x_2y_2}$ be radial grid diagrams constructed from Lagrangian grid diagrams $G_{x_1y_1}$ and $G_{x_2y_2}$, and let $\gamma_1 : \theta_1 \mapsto (x_1(\theta_1),y_1(\theta_1))$ and $\gamma_2 : \theta_2 \mapsto (x_2(\theta_2),y_2(\theta_2))$ the immersed loops defined by these radial grid diagrams.  Then the immersed torus $f: T^2 \rightarrow B^2$:
$$f(\theta_1,\theta_2) = (x_1(\theta_1),y_1(\theta_1),x_2(\theta_2),y_2(\theta_2),\sqrt{1-x_1^2-y_1^2-x_2^2-y_2^2},0)$$
lifts to an immersed Legendrian torus $\tilde{f}: T^2 \rightarrow S^5 \subset \CC^3$:
$$\tilde{f}(\theta_1,\theta_2) = e^{i t(\theta_1,\theta_2)}(x_1(\theta_1),y_1(\theta_1),x_2(\theta_2),y_2(\theta_2),\sqrt{1-x_1^2-y_1^2-x_2^2-y_2^2},0),$$
whose cone in $\CC^3$ is Lagrangian.
\end{theorem}

Consider the example shown in Figure~\ref{fig:radialGridPair}.  The dark shaded region of the first diagram has area $3\cdot \frac{\pi}{75}$, as does the light shaded region.  However, if we orient the two regions, using the orientation of the knot along the boundary of each, we see that the two regions have opposite orientation.  The result of this is that when computing the change in $t$, the contributions of each region will have opposite sign. Since each contribution is equal in magnitude, the total change in $t$ when traversing the entire knot is 0.  Moreover, observe that the difference in the $t$ coordinates at the crossing is $3 \cdot \frac{2 \pi}{75}$.  Similarly, one can see that the total change in $t$ for the second grid diagram is 0, and that the difference in the $t$ coordinates at each crossing is $2\cdot \frac{2 \pi}{75}$.

\begin{remark}
In general, beginning with two Lagrangian grid diagrams, converting to radial grid diagrams, and lifting, one produces an immersed torus, and hence an immersed Lagrangian cone.  To get an embedded torus, and hence an embedded Lagrangian cone, one must check to see that the product lift condition is satisfied by the pair of Lagrangian grid diagrams (cf. Section 4 of \cite{BaldMcCar2}).  This amounts to checking that Condition 2 of Theorem~\ref{thm:A} is satisfied.  The pair of radial grid diagrams shown in Figure~\ref{fig:radialGridPair} satisfies the product lift condition, as one may check.
\end{remark}

\begin{figure}
\includegraphics[scale=.3]{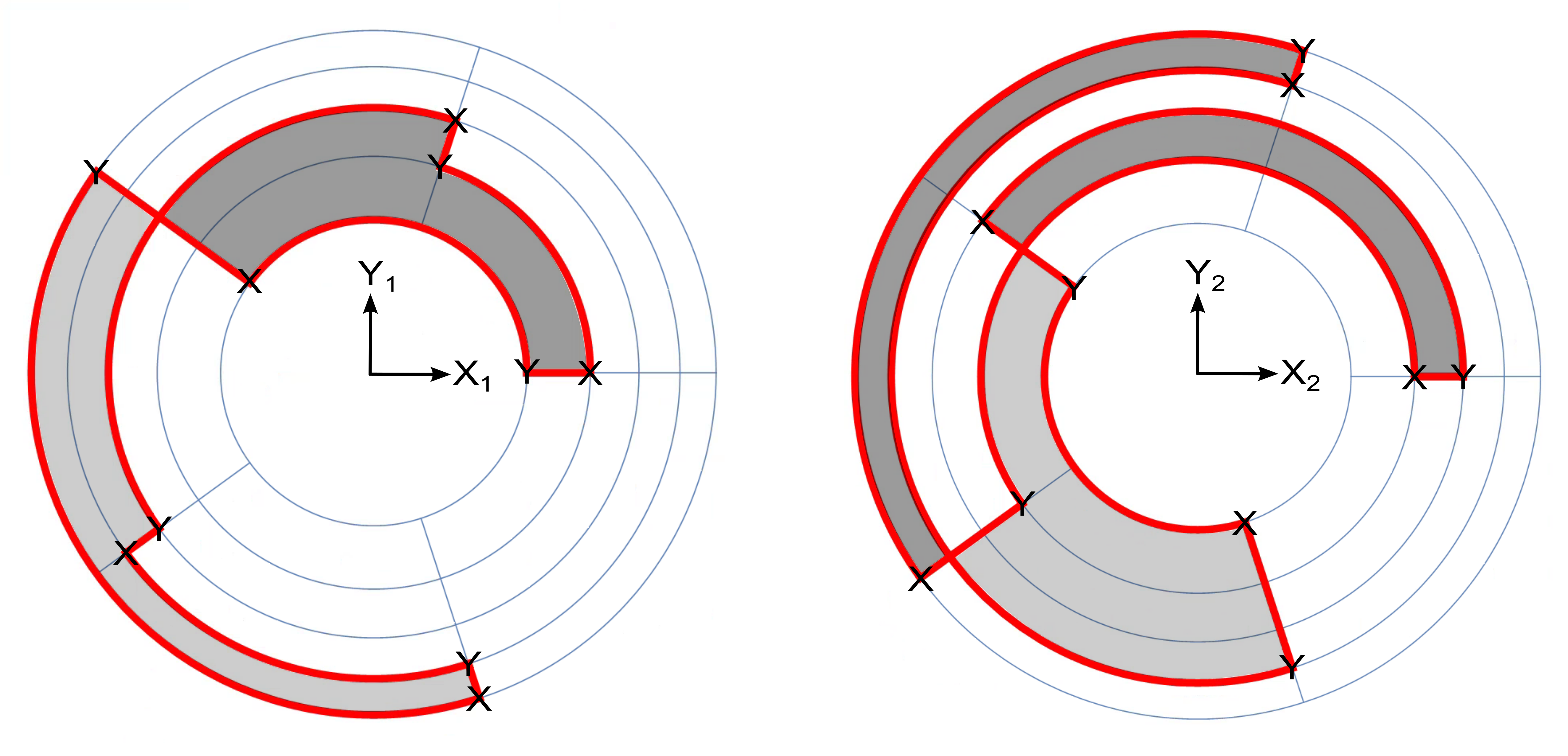}
\caption{A pair of loops that give rise to a Lagrangian cone.}
\label{fig:radialGridPair}
\end{figure}

\begin{remark}
\label{rem:smoothing}
In Proposition 3.4 of \cite{BaldMcCar2} it was shown that the the immersion determined be a Lagrangian grid diagram could be smoothed in such a way as to ensure that the lift of the smoothed immersion is $C^0$-close to the lift of the original immersion, and that any two smoothings, sufficiently close to the original immersion, would have Legendrian isotopic lifts.  The proof of that proposition depended only on the fact that the lift was determined by a net-area calculation. Since the same is true in this setting, the proof may be adapted to to this situation, to produce a smoothly embedded Lagrangian cone.
\end{remark}

The family of examples produced here is specific to the case $n=3$, but only because the Lagrangian hypercube diagrams are constructed, at this time, only in dimension $4$.  Yet, it is clear that Lagrangian hypercube diagrams may be generalized to produce Lagrangian immersions $f: T^{n-1} \rightarrow B^{n-1}$, leading to the following question:

\begin{question}
\label{qu:hypercube}
What types of Lagrangian cones may be produced as lifts of Lagrangian hypercube diagrams in even dimensions greater than $4$?
\end{question}

\subsection{Examples constructed from radial hypercube diagrams}
\label{sec:HypercubeCones2}
In the previous example, beginning with a pair of Lagrangian grid diagrams meant that for any loop on the immersed torus in $B^2$, in the lift, the net change in $t$ is $0$.  However, this is more restrictive than necessary, since we still obtain a well-defined lift provided that the net change in $t$ along any loop downstairs is an integer multiple of $2\pi$.  In fact, we may relax the conditions of the previous example a bit more, as follows.

Let $G_{x_1y_1}$ and $G_{x_2y_2}$ be two grid diagrams, and construct radial grid diagrams $\hat{G}_{x_1y_1}$ and $\hat{G}_{x_2y_2}$ by placing markings as in the previous example.  However, to obtain an immersed loop from the diagram, we follow a slightly different procedure.  Along each radial column, join the markings as in the original grid diagram.  In each circular row, there are two arcs oriented from $X$ to $Y$.  Choose one of the two oriented arcs in each row.  Figure~\ref{fig:radialGrid3} shows one example of a grid diagram, with a particular choice of connections made in each row. Thus to a given grid diagram of size $n$, there are $2^n$ distinct, immersed loops that correspond to it by following this procedure.

\begin{figure}
\includegraphics[scale=.3]{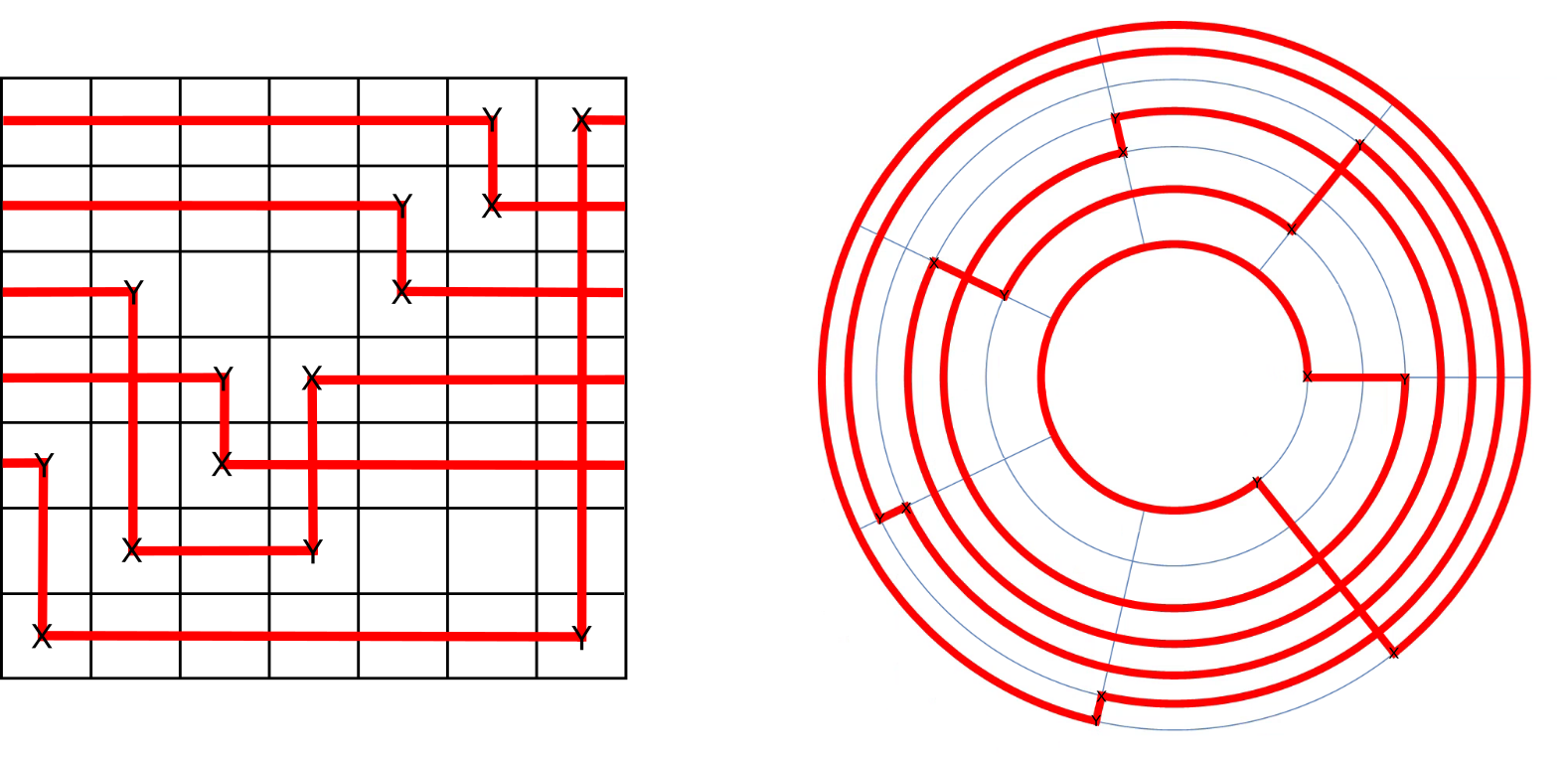}
\caption{A $7\times 7$ radial Lagrangian grid, with the associated grid diagram from which it is constructed.}
\label{fig:radialGrid3}
\end{figure}

\begin{theorem}
Let $\hat{G}_{x_1y_1}$ and $\hat{G}_{x_2y_2}$ be radial grid diagrams and let $\gamma_1 : \theta_1 \mapsto (x_1(\theta_1),y_1(\theta_1))$ and $\gamma_2 : \theta_2 \mapsto (x_2(\theta_2),y_2(\theta_2))$ the immersed loops defined by these radial grid diagrams, together with a choice of oriented circular arcs.  

Suppose that $\sum_{i=1}^n a_i r_i^2 = 2\pi k_1$, where $a_i$ is the angle subtended by the chosen arc in row $i$ of $\hat{G}_{x_1,y_1}$, $r_i$ is the radius of the corresponding circle, and $k_1 \in \ZZ$.  Similarly assume that $\sum_{i=1}^n b_i r_i^2 = 2\pi k_2$, where $b_i$ is the angle subtended by the chosen arc in row $i$ of $\hat{G}_{x_2,y_2}$, $r_i$ is the radius of the corresponding circle, and $k_2 \in \ZZ$.  Then the immersed torus $f: T^2 \rightarrow B^2$:
$$f(\theta_1,\theta_2) = (x_1(\theta_1),y_1(\theta_1),x_2(\theta_2),y_2(\theta_2),\sqrt{1-x_1^2-y_1^2-x_2^2-y_2^2},0)$$
lifts to an immersed Legendrian torus $\tilde{f}: T^2 \rightarrow S^5 \subset \CC^3$:
$$\tilde{f}(\theta_1,\theta_2) = e^{i t(\theta_1,\theta_2)}(x_1(\theta_1),y_1(\theta_1),x_2(\theta_2),y_2(\theta_2),\sqrt{1-x_1^2-y_1^2-x_2^2-y_2^2},0),$$
where $t$ is defined as in Theorem~\ref{thm:A}, whose cone in $\CC^3$ is Lagrangian.
\end{theorem}

\begin{proof}
The proof follows from Theorem~\ref{thm:A} together with the observations of Theorem~\ref{thm:hypercubeLift} that the change in $t$ may be interpreted as a net-area calculation.  The condition that $\sum_{i=1}^n a_i r_i^2 = 2\pi k_1$ and $\sum_{i=1}^n b_i r_i^2 = 2\pi k_2$ guarantees that the net-area of the loops determined by $\hat{G}_{x_1y_1}$ and $\hat{G}_{x_2y_2}$, is a multiple of $2\pi$ and hence, each loop lifts to a loop that wraps around the fiber $k_1$ or $k_2$ times.  
\end{proof}

 \begin{figure}
\includegraphics[scale=0.3]{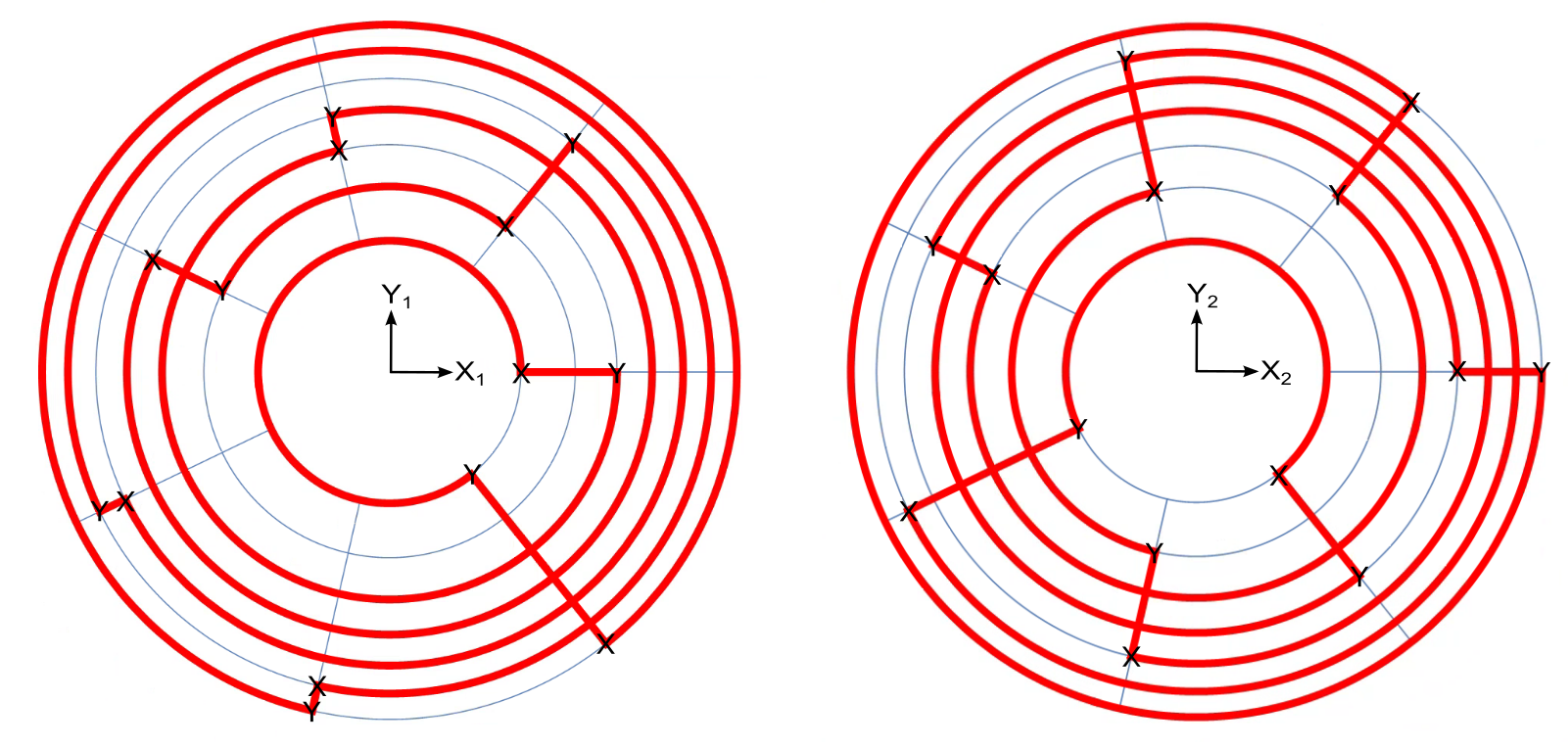}
\caption{A pair of $7\times 7$ radial grid diagrams that give rise to a Lagrangian cone.}
\label{fig:radialGrid4}
\end{figure}

The two radial grid diagrams shown in Figure~\ref{fig:radialGrid4} determine an immersion that lifts to a torus whose cone is Lagrangian.  An easy net-area calculation shows that the cone is embedded, since the two diagrams satisfy the product lift condition (cf. Section 4 of \cite{BaldMcCar2}).  Moreover, the lift has the property that each diagram lifts to a loop that wraps once around the fiber.

\begin{remark}
The pair of grid diagrams chosen at the beginning determine a structure, similar to a hypercube diagram, which we will refer to as a \emph{radial Lagrangian hypercube diagram}.  
\end{remark}

\begin{remark}
Remark~\ref{rem:smoothing} applies in this situation as well, allowing us to produce smooth Lagrangian cones using radial Lagrangian hypercube diagrams.
\end{remark}

In light of Example~\ref{ex:HL}, it is natural to ask which Lagrangian hypercube diagram gives rise to the Harvey-Lawson cone.  Note that the immersion given in Example~\ref{ex:HL} does not readily admit the structure of a Lagrangian hypercube diagram.  It has only two double point circles, neither of which intersect, while any Lagrangian hypercube diagram must contain double point circles that intersect (since each Lagrangian grid diagram used to define a Lagrangian hypercube diagram must contain crossings, each of which produces a double point circle in the product).  Nevertheless, it seems likely that there is a Lagrangian hypercube representation of the Harvey-Lawson cone,  hence:

\begin{conjecture}
\label{con:hypercubeHL}
There exists a radial Lagrangian hypercube diagram, whose associated Lagrangian cone in $\CC^3$ is isotopic to the Harvey-Lawson cone.
\end{conjecture}

Lastly, we can ask the same question here as we did with Lagrangian hypercube diagrams (cf. Question~\ref{qu:hypercube}).

\begin{question}
\label{qu:hypercube2}
What types of Lagrangian cones may be produced as lifts of radial Lagrangian hypercube diagrams in even dimensions greater than $4$?
\end{question}

\section{The Main Theorem}
\label{sec:thmB}

While Theorem~\ref{thm:A} applies only to immersions into a unit ball, $B^{n-1}\subset \CC^{n-1}$, thought of as a single chart of $\CC P^{n-1}$, the good news is that it can be generalized to any immersion $f: \Sigma^{n-1} \rightarrow \CC P^{n-1}$ so that the lifting process works in much the same way as it does in Theorem~\ref{thm:A}.  This is the content of the Main Theorem below.  We build up to the Main Theorem through a series of computationally useful lemmas and definitions.  

Recall that the symplectic form associated with the Fubini-Study metric is, in coordinates $z=(z_1,...,z_n)$ of $\pi_{\CC^*}: \CC^n\setminus \{0\} \rightarrow \CC P^{n-1}$, given by 
\begin{equation}
\label{eqn:Fubini}
\pi_{\CC^*}^*(\omega_{FS}) = \frac{i}{2} \cdot \frac{1}{|z|^4} \sum_{k=1}^n \sum_{j\neq k} \left( \overline{z_j}z_j dz_k \wedge d\overline{z_k} - \overline{z_j}z_k dz_j \wedge d\overline{z_k}  \right).
\end{equation}
The form $\omega_{FS}$ is the form induced upon $\CC P^{n-1}$ after quotienting by the invariant $\CC^*$ action.  It is easy to check that 
$$\int_{\CC P^1} \omega_{FS} = \pi.$$
and therefore $\frac{1}{\pi} \omega_{FS}$ is an integral symplectic form on $\CC P^{n-1}$.  Furthermore, for $i:S^{2n-1} \rightarrow \CC^n$, it is well-known that $\omega_{FS}$ is the unique form such that $i^*(\omega_0) = \pi^*(\omega_{FS})$ where $\pi: S^{2n-1} \rightarrow \CC P^{n-1}$ is the Hopf fibration and $\omega_0$ is the standard symplectic form on $\CC^n$, i.e. for $z_i = x_i + i y_i$,
$$\omega_0 = \frac{i}{2} \sum_{i=1}^n dz_i\wedge d\overline{z_i} = \sum_{i=1}^n dx_i \wedge dy_i.$$

As mentioned above, the usual homogeneous, holomorphic coordinate system on $\CC P^{n-1}$ is not suitable for our purposes.  Instead, we use the hemispherical coordinate system:

\begin{definition}
Let $B_i \subset \CC^{n-1}$ be the open unit ball and define coordinate charts $\psi_i: B_i \rightarrow \CC P^{n-1}$, $j=1,...,n$, given by
$$\psi_i(z_1,...,z_{i-1},z_{i+1},...,z_{n-1}) = [z_1: ...: z_{i-1} : \sqrt{1-|z|^2} : z_{i+1} : ... : z_{n}].$$
The charts, $(B_i,\psi_i)$ are called \emph{hemispherical charts}.  
\end{definition}

Note that we are numbering the $z_i$'s in terms of $\CC^n$ instead of $\CC^{n-1}$.  For example $z\in B_2 \subset \CC^2$ is defined by $z=(z_1,z_3)$ and is mapped to $\CC P^3 = \CC^3\setminus \{0\} / \CC^*$ as $\psi_2 (z_1,z_3) = [z_1: \sqrt{1-|z|^2}:z_3]$ where $|z|^2 = |z_1|^2 + |z_3|^2$.  We will also use the hat symbol to denote removing a term.  Hence $z = (z_1, z_3)$ could also be written as $z = (z_1,\hat{z_2},z_3)$ to simplify notation. 

Also, we use $U_i$ to refer to the image of $B_i$ in $\CC P^{n-1}$, i.e. $U_i = \psi_i(B_i)$.  The name of the system obviously follows from the fact that the image of each chart is the image of a hemisphere in $S^{2n-1} \subset \CC^n$ via the Hopf fibration $\pi: S^{2n-1} \rightarrow \CC P^{n-1}$.  

The hemispherical charts, $\psi_i$, are not holomorphic with respect to the natural complex structure on $\CC P^{n-1}$.  However, they do have one very nice property:  the $\psi_i$'s are Darboux charts on $\CC P^{n-1}$.

\begin{lemma}
\label{lem:w0wfs}
If $\omega_0$ is the standard symplectic form on $B\subset \CC^{n-1}$ then
$$\omega_0=\psi_i^*(\omega_{FS}).$$
\end{lemma}

\begin{proof}
Observe that in homogeneous coordinates, Equation~\ref{eqn:Fubini} translates into:
$$\tilde{\omega}_{FS} = \frac{i}{2|z|^4}(\overline{z_2}z_2 dz_1\wedge d\overline{z_1} - \overline{z_2} z_1 dz_2 \wedge d\overline{z_1} + \overline{z_1}z_1 dz_2\wedge d\overline{z_2} - \overline{z_1} z_2 dz_1 \wedge d\overline{z_2}).$$
Observe that in $B_1$, $z_1 = \sqrt{1-|z_2|^2}$.  Using this observation, and changing to real coordinates, observe that in hemispherical coordinates, $\tilde{\omega}_{FS} = dx_2 \wedge dy_2$, which is $\omega_0$ in the chart $B_1$.  The general calculation is similar.
\end{proof}

Before moving on, we can characterize the sets $U_i$ and point out that the $\psi_i$'s are a chart system (all points of $\CC P^{n-1}$ are in at least one chart).  Let $[z_1:...:z_n] \in \CC P^{n-1}$.  At least one coordinate is non-zero, say $z_i \neq 0$.  In the pre-image of the quotient map for $\CC P^{n-1} = \CC^n \setminus{0} / \CC^*$, the point $(z_1,...,z_n)$ is equivalent to $\frac{\overline{z_i}}{|z_i| |z|}(z_1,...,z_n)$ where $|z| = \sqrt{|z_1|^2+...+|z_n|^2}$.  Therefore $[z_1:...:z_n]\in U_i$ and $U_i=\{[z_1:...:z_n] \ | \ z_i\neq 0\}$.  Thus, the hemispherical chart system allows us to work with $f(\Sigma)|_{U_i} \subset B_i$ using the standard symplectic form $\omega_0$.

Hemispherical charts also trivialize the Hopf fibration over $\CC P^{n-1}$.  In the diagram,
\begin{center}
\begin{tikzcd}
B_i \times S^1  \arrow[r, "\Psi_i"]  \arrow[d,"\pi"]	&  S^{2n-1} \arrow[d, "\pi_{S^1}"]  \arrow[r,hook]  & \CC^n \\
B_i  \arrow[r, "\psi_i"]  					&  \CC P^{n-1}. &  \\
\end{tikzcd}
\end{center}
$B_i \times S^1$ is a trivialization of the $S^1$-bundle, $\pi:  S^{2n-1} \rightarrow \CC P^{n-1}$, given by 
$$\Psi_i(z,e^{it})=e^{it} \left(z_1, ..., z_{i-1},\sqrt{1-|z|^2},z_{i+1},...,z_n\right)\in S^{2n-1} \subset \CC^n.$$

It is easy to see that the diagram commutes and that $\Psi_i$ gives a trivialization of the Hopf fibration over $U_i \subset \CC P^{n-1}$.  

As mentioned before, there is a natural contact form $\alpha$ on the unit sphere $S^{2n-1}$ in $\CC^n$.  Given $z=(z_1,...,z_n)\in \CC^n$ where $z_i=x_i+i y_i$ and $\omega_0 = \frac{i}{2} \sum_{i=1}^n dz_i\wedge d\overline{z_i} = \sum_{i=1}^n dx_i \wedge dy_i$, the form
$$\alpha_0=\frac{1}{2}\left( \sum_{i=1}^n x_i dy_i - y_i dx_i\right)$$
is a contact form when restricted to $S^{2n-1}$.  Set $\alpha = \alpha_0 |_{S^{2n-1}}$.  Equipped with this contact form, $(S^{2n-1},\alpha)$ is a contact manifold.

We collect a few facts about $\alpha$---partly to set notation for the reader, and partly to justify choices and conventions used throughout this paper.

\begin{lemma}
\label{lem:ReebField}
For $z=(z_1,...,z_n)\in \CC^n\setminus{0}$ where $z_i=x_i+i y_i$, let $N_z = x_1\frac{\partial}{\partial x_1} + y_1 \frac{\partial}{\partial y_1} + ... + x_n \frac{\partial}{\partial x_n} + y_n \frac{\partial}{\partial y_n}$ be the outward pointing normal vector field for any sphere of radius $r>0$, centered at the origin in $\RR^{2n}$, and $T_z = x_1 \frac{\partial}{\partial y_1} - y_1 \frac{\partial}{\partial x_1} + ... + x_n \frac{\partial}{\partial y_n} - y_n \frac{\partial}{\partial x_n}$ be the vector field that generates the Hopf fibration $\pi : S^{2n-1} \rightarrow \CC P^{n-1}$.  Then the following are facts about $\alpha_0$ and the contact form $\alpha$:
\begin{enumerate}
\item The form $\alpha_0$ is equal to $\iota_{\frac{1}{2}N_z}\omega_0$ when $|z| = 1$.
\item The form $\alpha_0$ also satisfies $\alpha_0 (k T_z) = \frac{k}{2} |z|^2$ for $k$ a constant, and $\iota_{T_z} d\alpha_0 = \iota_{T_z} \omega_0 = -\sum_{i=1}^n (x_i dx_i + y_i dy_i)$.
For any vector $v\in T_z S_r^{2n-1}$ (for sphere of radius $r=|z|$
$$\iota_{T_z} d\alpha_0 (v) = -\langle N_z,v \rangle=0$$
where $\langle \;, \; \rangle$ is the usual inner product on $\RR^{2n}$.  Therefore the vector field $R$, defined by $R=2T_z$ when restricted to $|z|=1$, is the reeb vector field of $\alpha$, i.e. $\alpha(R) = 1$ and $d\alpha(R,\cdot) = 0$.
\item Since $i^*(\omega_0) = \pi^*(\omega_{FS})$ and $d\alpha_0 = \omega_0$, $\frac{i}{\pi} \alpha$ is the connection one-form of the integral cohomology class $[\frac{1}{\pi} \omega_{FS}]$.  
\end{enumerate}
\end{lemma}

We use $\alpha$ for $\eta$ in Theorem~\ref{thm:GenLift} to find $\Psi_i^*(\alpha)$ in the trivialization $B_i\times S^1$ with coordinates $(z,e^{it})$.  

\begin{lemma}
\label{lem:TrivAlpha}
Let $B_j \subset \CC^{n-1}$ be the unit ball with coordinates $z=(z_1,...,z_{j-1},z_{j+1},...,z_n)$.  For a chart $\psi_j: B_j\rightarrow \CC P^{n-1}$ and trivialization $\Psi_j : B_j \times S^1 \rightarrow S^{2n-1}$ given by $\Psi_j(z,e^{it}) = e^{it} (z_1,...,z_{j-1},\sqrt{1-|z|^2},z_{j-1},...,z_n)$, then 
$$\Psi_j^*(\alpha) = \frac{1}{2}\left( dt + 2\alpha_0 \right)$$
where $\alpha_0$ is the form defined above on $B_j \subset \CC^{n-1}$. 

In polar coordinates,
$$\Psi_j^*(\alpha) = \frac{1}{2} \left(dt + r_1^2 d\theta_1 + ... + \widehat{r_j^2 d\theta_j} + ... + r_n^2 d\theta_j \right).$$
\end{lemma}

Note that the $\alpha_0$ defined on $B_j$ has no $z_j$ term of the form $(x_j dy_j - y_j dx_j)$ since $z \in B_j$ has coordinates $z= \left(z_1,...,\hat{z_j},...,z_n \right)$.  

\begin{proof}
The calculation is easiest in polar coordinates.  In $\CC^n$, $\alpha_0 = \frac{1}{2} \left(\sum_{i=1}^n r_i^2 d\theta_i \right)$, and 
$$\Psi_j^*(\alpha) = \frac{1}{2} \left[ \sum_{\substack{i=1 \\ i \neq j}}^n r_i^2 d(\theta_i + t) + \left(\sqrt{1-\sum_{\substack{i=1 \\ i \neq j}}^n r_i^2} \right) ^2 dt \right]$$
This reduces to
\begin{eqnarray*}
\Psi_j^*(\alpha) & = & \frac{1}{2} \left[ \sum_{\substack{i=1 \\ i \neq j}}^n r_i^2 d\theta_i +  \sum_{\substack{i=1 \\ i \neq j}}^n r_i^2 dt + \left(1 - \sum_{\substack{i=1 \\ i \neq j}}^n r_i^2 \right)dt \right]\\
 & = & \frac{1}{2} \left[ \sum_{\substack{i=1 \\ i \neq j}}^n r_i^2 d\theta_i + dt \right]\\
\end{eqnarray*}
\end{proof}

Thus we can take $\tau$ in Theorem~\ref{thm:GenLift} to be the $1$-form $-2\alpha_0 \in \Omega^1 (B_j)$.  In each chart $B_j\times S^1$, label $\tau_j = -2\alpha_0$,  note that the transition map $\Psi_{kj} : B_j \times S^1 \rightarrow B_k \times S^1$ takes 
\begin{equation}
\label{eqn:YY}
\Psi_{kj}^* \left(\frac{1}{2}(dt - \tau_k)\right) = \frac{1}{2}(dt - \tau_j).
\end{equation}

This result follows from the next lemma.

\begin{lemma}
\label{lem:Transition}
Let $B_j$ be the unit ball in $\CC^{n-1}$ with coordinates $z=(z_1,...,z_{j-1},z_{j+1},...,z_n)$.  Let $\Psi_j : B_j \times S^1 \rightarrow S^{2n-1} \subset \CC^n$ given by $\Psi_j(z,e^{it}) = e^{it} ( z_1, ..., z_{j-1}, \sqrt{1-|z|^2}, z_{j+1}, ..., z_n )$ where $|z|=|z_1|^2 + ... + \widehat{|z_j|^2} + ... + |z_n|^2$.
For $k\neq j$, the map
$$\Psi_{kj} : B_j\setminus \{z_k = 0\} \times S^1 \rightarrow B_k \setminus \{z_j = 0\} \times S^1$$
defined by $\Psi_{kj} = \Psi_k^{-1} \circ \Psi_j$ is given by the map 
$$\Psi_{kj} (z,e^{it} ) = \left( z_1 \frac{\overline{z_k}}{|z_k|}, ... , z_{k-1} \frac{\overline{z_k}}{|z_k|}, |z_k|, z_{k+1} \frac{\overline{z_k}}{|z_k|}, ... , z_{j-1} \frac{\overline{z_k}}{|z_k|} , \frac{\overline{z_k}}{|z_k|} \sqrt{1-|z|^2} , z_{j+1} \frac{\overline{z_k}}{|z_k|} , ... , z_n \frac{\overline{z_k}}{|z_k|},  e^{it} \frac{z_k}{|z_k|}\right)$$
In polar coordinates, 
\begin{multline*}
\Psi_{kj}(r_1,\theta_1,...,\hat{r_j},\hat{\theta_j},...,r_n,\theta_n, t) = \\
\left(r_1, \theta_1-\theta_k, r_2, \theta_2 - \theta_k, ... , r_{j-1} , \theta_{j-1} - \theta_k, \sqrt{1- \sum_{\substack{i=1 \\ i \neq j}}^n r_i^2} , -\theta_k, r_{j+1}, \theta_{j+1}, ... , r_n, \theta_n - \theta_k, t + \theta_k \right)\\
\end{multline*}
\end{lemma}

\begin{proof}
We show the calculation for $B_2, B_3 \subset \CC^3$.  The general case is similar.  The maps 
$$\Psi_2 : B_2 \times S^1 \rightarrow S^7 \subset \CC^4,$$
$$\Psi_2 (z_1,z_3,z_4,e^{it}) = e^{it} \left( z_1, \sqrt{1-|z|^2}, z_3,z_4\right),$$
and
$$\Psi_3: B_3\times S^1 \rightarrow S^7 \subset \CC^4,$$
$$\Psi_3(w_1,w2,w4,e^{it}) = e^{it} \left( w_1, w_2, \sqrt{1-|w|^2},w_4 \right)$$
give rise to $\Psi_{32}: B_2\setminus \{z_3 = 0\}\times S^1 \rightarrow B_3 \setminus \{w_2 = 0 \} \times S^1$ via $\Psi_3^{-1} \circ \Psi_2$.  By multiplying by 1 appropriately, 
\begin{eqnarray*}
\Psi_2(z_1, z_3, z_4, e^{it} ) & = & e^{it} \left( z_1, \sqrt{1-|z_1|^2 - |z_3|^2 - |z_4|^2}, z_3, z_4\right)\\
					    & = & e^{it} \left( \frac{z_3}{|z_3|} \frac{\overline{z_3}}{|z_3|} \right) \left( z_1, \sqrt{1-|z_1|^2 - |z_3|^2 - |z_4|^2}, z_3, z_4 \right)\\
					    & = &  e^{it} \left( \frac{z_3}{|z_3|} \right) \left( z_1\frac{\overline{z_3}}{|z_3|} ,\frac{\overline{z_3}}{|z_3|}  \sqrt{1-|z_1|^2 - |z_3|^2 - |z_4|^2}, z_3\frac{\overline{z_3}}{|z_3|} , z_4\frac{\overline{z_3}}{|z_3|}  \right)\\
					    & = & \Psi_3 \left( w_1, w_2, w_4, e^{it'} \right) 
\end{eqnarray*}
where $w_1 = z_1 \frac{\overline{z_3}}{|z_3|}$, $w_2 = \frac{\overline{z_3}}{|z_3|} \sqrt{1-|z|^2}$, $w_4 = z_4 \frac{\overline{z_3}}{|z_3|}$, and $e^{it'} = e^{it} \frac{z_3}{|z_3|}$.
It is easy to check that $(w_1, w_2, w_4, e^{it}) \in B_3 \setminus \{w_2 = 0 \}$, and that $\sqrt{1-|w_1|^2 - |w_2|^2 - |w_4|^2} = |z_3|$ and $e^{it} \frac{z_3}{|z_3|} \in S^1$ as desired.
\end{proof}

\begin{remark}
\label{rem:RelTrans}
The formula for $\Psi_{kj}$ also gives the formula for $\psi_{kj} : B_j \setminus \{z_k = 0 \} \rightarrow B_k \setminus \{ z_j = 0 \}$ for $\psi_{kj} = \psi_k^{-1} \circ \psi_j$ by looking at the $z$ coordinates of $(z,e^{it})$.  
\end{remark}

In summary, given a Lagrangian immersion $f: \Sigma \rightarrow \CC P^{n-1}$ and $V_j = f(\Sigma) \bigcap B_j$, we can work with $V_j \subset B_j$ using 
\begin{itemize}
\item the standard symplectic form $\omega_0$ on $B_j \subset \CC^{n-1}$,
\item the standard $1$-form $\tau_j = - 2\alpha_0$ on $B_j \subset \CC^{n-1}$,
\end{itemize}
and patch the $V_j$'s together using the transition maps $\psi_{kj} : B_j \rightarrow B_k$ given by $\psi_{kj} = \psi_k^{-1} \circ \psi_j$.

In practice, this allows us to do integration and other calculations in the $B_j$'s using standard forms in each instead of working with homogeneous coordinates and $\omega_{FS}$ in $\CC P^{n-1}$.

This chart system also gives us new ways to build examples of Lagrangian immersions by first working with piecewise linear submanifolds in each ball $B_j$, pasting the pieces together, and then smoothing the result (as is done with Lagrangian hypercubes in \cite{Bald} and Section 3 of \cite{BaldMcCar2}).

\subsection{The Main Theorem}

The Main Theorem puts the separate pieces in the previous sections together into one result.  First, we need an explicit way to calculate integrals along paths in $f(\Sigma)$.  

Let $f:\Sigma\rightarrow \CC P^{n-1}$ be a Lagrangian immersion and let $\gamma: I \rightarrow \Sigma$ be a path.  In order to define the lift, we need to define a map $t : I \rightarrow \RR / 2\pi \ZZ$, which we do in pieces.  Split the interval $I$ into subintervals
$$I=\bigcup_{k=0}^{m-1} [s_k,s_{k+1}]$$
where $0=s_0<s_1<...<s_{m-1}<s_m=1$ such that $f(\gamma([s_k,s_{k+1}])) \subset B_j$ for some $j\in \{1,...,n\}$ (after identifying $B_j$ with $U_j$ using $\psi_j$).  Index the $B_j$'s by $j_k$ so that $f(\gamma([s_k,s_{k+1}])) \subset B_{j_k}$ where $j_k$ is the index of the chart in which $\gamma([s_k,s_{k+1}])$ is contained.  Let $x_k = \gamma(s_k)$ so that $x_0=\gamma(s_0)$ and $x_m = \gamma(s_m)$.  Also, for convenience, use the notation $(z)_k$ to stand for the $z_k$ coordinate of $z\in B_j$. (If $z\in B_3 \subset \CC^3$ such that $z=(z_1,z_2,z_4)$ then $(z)_4 = z_4$.) 

Since $f(\gamma([s_0,s_1])) \subset B_{j_0}$, we can integrate $\tau_{j_0} = -2\alpha_0$ (cf. Equation~\ref{eqn:YY}) along the path $f(\gamma([s_0,s_1]))$.  Define $t_0 : [s_0,s_1] \rightarrow \RR / 2\pi \ZZ$ by 
$$t_0(s) = \left( \int_0^s \tau_{j_0}\left(( f\circ \gamma)'(u)\right)du\right) \text{ mod }2\pi$$
where $t_0(0) = 0$.

For $s\in [s_0,s_1]$ and $t(0) = a$, we can write 
$$t(s) = t_0(s)+a.$$
The point $\left( f(\gamma(s_1)),e^{it(s_1)} \right) \in B_{j_0} \times S^1$ also lives as a point $\Psi_{j_1j_0} \left( f(\gamma(s_1)), e^{it(s_1)} \right) \in B_{j_1} \times S^1$.   Define $\Psi_{j_1j_0}(t(s_1)) \in \RR / 2\pi \RR$ to be the argument of the $S^1$ component of this map in $B_{j_1}\times S^1$.  We can also define the point $\psi_{j_1j_0}(f(\gamma(s_1))\in B_{j_1}$ as the $B_{j_1}$ component of $B_{j_1} \times S^1$ (see Remark~\ref{rem:RelTrans}).  

\begin{lemma}
\label{lem:Trans}
When $t(s_k)$ is defined for $\left( f(\gamma(s_k)), e^{it(s_k)}\right) \in B_{j_{k-1}}$, then $\Psi_{j_k j_{k-1}}(t(s_k)) = t(s_k) + arg\left( \psi_{j_k j_{k-1}} (f(\gamma(s_k)))_{j_k}\right)$. 
\end{lemma}

\begin{proof}
See Lemma~\ref{lem:Transition}.  
\end{proof}

We can now continue the integration in $B_{j_1}$:  Define $t_1: [s_1,s_2] \rightarrow \RR / 2\pi \ZZ$ by $t_1(s_1)=0$ and 
$$t_1(s) = \left(\int_{s_1}^{s} \tau_{j_1}((f\circ\gamma)'(u))du\right) \text{ mod } 2\pi.$$
Hence we can write $t(s)$ for $s\in [s_1, s_2]$ as 
$$t(s) = t_1(s) + \Psi_{j_1 j_0} (t_0(s_1) + a).$$
Induct on $k$ to integrate the $\tau_{j_k}$'s over the entire path:

\begin{definition}
\label{def:lifitingIntegral}
Let $[0,1]\xrightarrow{\gamma} \Sigma \xrightarrow{f} \CC P^{n-1}$ and suppose there exists an increasing sequence $0 = s_0 < s_1 < ... < s_{m-1} < s_m=1$ such that $f(\gamma([s_k,s_{k+1}])) \subset B_{j_k}$ for $j_k \in \{1,...,n\}$ and $f(\gamma(s_k)) \neq 0$ and $f(\gamma(s_{k+1})) \neq 0$ for all $0\leq k \leq m$.  Assume $t(0) = a$ and define the lifting integral to be 
\begin{multline*}
\Gamma \int_\gamma \tau := \\
\left[ t_{m-1}(s_m) + \Psi_{j_{m-1} j_{m-2}} \left( \cdots \left( t_3(s_4) + \Psi_{j_3 j_2} \left( t_2(s_3) + \Psi_{j_2 j_1} \left( t_1 (s_2) + \Psi_{j_1 j_0} \left(t_0(s_1) + a \right) \right) \right) \right) \right) \right] \text{ mod } 2\pi\\
\end{multline*}
\end{definition}

In practice we usually need only $m=1$ or $m=2$ for most integrals.  Also, since 
$$\tau_{j_k} = - \sum_{\substack{i=1 \\ i \neq j_k}}^n \left( x_i dy_i - y_i dx_i \right)$$
and 
$$\omega_{FS}|_{B_{j_k}} =  \sum_{\substack{i=1 \\ i \neq j}}^n dx_i\wedge dy_i,$$
the calculations are easy to do in each chart.

We can now state the Main Theorem.

\begin{theorem*}
\label{thm:B}
Let $\Sigma$ be a closed, connected, smooth $(n-1)$-manifold, and $f:\Sigma \rightarrow \CC P^{n-1}$ a Lagrangian immersion with respect to the integral symplectic form $\frac{1}{\pi}\omega_{FS}$.  Let $\pi: S^{2n-1} \rightarrow \CC P^{n-1}$ be the principle Hopf $S^1$-bundle with connection $1$-form $\frac{i}{\pi} \alpha$ where $\alpha = i_0^* \left( \frac{1}{2} \sum_{i=1}^n x_i dy_i - y_i dx_i \right)$ for the identity map $i_0 : S^{2n-1} \rightarrow \CC^n$.  For each chart $\Psi_j: B_j \times S^1 \rightarrow S^{2n-1}$, there exists a $1$-form $\tau_j$ such that $\Psi_j^*(\alpha) = \frac{1}{2} (dt-\tau_j)$ where $\tau_j = - \sum_{\substack{i=1 \\ i \neq j}}^n (x_i dy_i - y_i dx_i)$.

If
\begin{enumerate}
\item $\Gamma \int_\gamma \tau = 0 \text{ mod }2\pi$ for all $[\gamma] \in H_1(\Sigma; \ZZ)$, and 
\item for all distinct points $x_1,...,x_k \in \Sigma$ such that $f(x_1) = f(x_j)$ for all $j\leq k$, and a choice of path $\gamma_j$ from $x_1$ to $x_j$ in $\Sigma$ for $2 \leq j \leq k$, the set $\left\{ \left( \Gamma \int_{f(\gamma_j)} \tau \right) \text{ mod } 2\pi \ | \ 2\leq j \leq k \right\}$ has $k-1$ distinct values, none of which are equal to $0$,

\end{enumerate}
then $f:\Sigma \rightarrow \CC P^{n-1}$ lifts to an embedding $\tilde{f}: \Sigma\rightarrow S^{2n-1}$ such that the image (the lift) $\tilde{\Sigma}$ is a Legendrian submanifold of $(S^{2n-1},\alpha)$.  Furthermore, the cone $c\tilde{\Sigma}$ is Lagrangian in $\CC^n$ with respect to the standard symplectic structure $\omega_0$.  

\end{theorem*}

\section{The trivial Lagrangian cone as a lift using the Main Theorem}

\begin{example}
\label{ex:trivial}
We already saw in Section~\ref{sec:LagSpec} how to obtain a trivial (special) Lagrangian cone, but, we can also construct this example using the Main Theorem, as a lift of a map $f: S^{n-1} \rightarrow \CC P^{n-1}$.  

Recall that the trivial cone is given by the map $\tilde{f}: \RR^n \rightarrow \CC^{n}$ where $(x_1,...,x_n) \mapsto (x_1 \eta_1, ... , x_n \eta_n)$, and $\eta = (\eta_1,...,\eta_n)$ is a complex vector with $\eta_j \neq 0$ for all $j$.  Clearly the trivial cone is a lift of the Lagrangian immersion $f: S^{n-1} \rightarrow \CC P^{n-1}$ given by $f(x_1,...,x_n) = [x_1 \eta_1 : ... : x_n \eta_n]$.

Observe that the set $\{ (x_1, ... , x_{n}) \subset \RR^n \ |  \ \sum_{k=1}^{n} |x_k \eta_{k,j}|^2 = 1 \}$ is an $(n-1)$-dimensional sphere, $S^{n-1}$ for any choice of complex vector $(\eta_{1,j},...,\eta_{n,j})$ (the reason for the $j$-subscript will be apparent shortly).  Moreover, we may cover $S^{n-1}$ by charts of the form $\phi_j^\pm : V_j^\pm \rightarrow S^{n-1}$ where $V_j^\pm = \{ (x_1,...,x_{j-1}, \hat{x_j}, x_{j+1},...,x_n)\in \RR^{n-1} \ | \ \sum_{\substack{k=1 \\ k \neq j}}^{n} |x_j \eta_{k,j}|^2 < 1 \}$, and the sign indicates which hemisphere is being covered.  Within each chart, after identifying $V_j^\pm$ with $\phi_j^\pm(V_j^\pm)$, we may write write $f(x)$ as $f_j^\pm(x)$ where $f_j ^\pm: V_j^\pm \rightarrow \CC P^{n-1}$ is given by the following:
$$f_j^\pm(x_1,...,x_{j-1}, \hat{x_j},x_{j+1},...x_n) = \left[ x_1 \eta_{1,j} : ... : x_{j-1} \eta_{j-1,j} : ^\pm \sqrt{1-\sum_{\substack{k=1 \\ k \neq j}}^{n} |x_k \eta_{k,j}|^2} : x_{j+1} \eta_{j+1,j} : ... : x_n \eta_{n,j} \right],$$
where $\eta_{k,j} = \eta_k \frac{\overline{x_j \eta_j}}{|x_j \eta_j|}$.

Since $H_1(S^2,\ZZ)$ is trivial, the first condition of the Main Theorem is automatically satisfied.  Moreover, $f_i^\pm$ is clearly an embedding on $V_i^\pm$, so within each chart the second condition is satisfied.  However, observe that after patching these maps together, the antipodal points of $S^{n-1}$ are the only ones identified by $f$ (in fact, the image of $f$ is a copy of $\RR P^{n-1}$).  A simple calculation shows that the integral $\Gamma \int_\gamma \tau \neq 0$ for any path $\gamma$ from $x_0 \in V_i^\pm$ to its antipode, $-x_0 \in V_i^\mp$:  within each chart, the integral is 0, but when transitioning between charts, we pick up a rotation of the $S^1$-factor.  Hence, the Main Theorem guarantees the existence of an embedded lift, $\tilde{f} : S^{n-1} \rightarrow S^{2n-1} \subset \CC^n$ such that the cone is Lagrangian in $\CC^n$.  Moreover, our discussion above clearly identifies this as a Lagrangian $\RR^n \subset \CC^n$, which is the trivial cone.

The trivial cone intersects $S^{2n-1}$ in a Legendrian $(n-1)$-sphere that projects down to a copy of $\RR P^{n-1}$ via a 2-to-1 map (the quotient by the antipodal map).  This is inconvenient when one wishes to compute Legendrian contact homology.   However, we can perturb $f$ through a family of functions $f_t$ so that the image of the lift, $\tilde{f_1}$, is a copy of $S^{n-1}$ having only transverse double points when projected down to $\CC P^{n-1}$.  

For simplicity, we write down the perturbation in the case where $n=3$, and the $\eta = (1,...,1)$.  Choose $\epsilon \geq 0$, and perturb each hemisphere of $S^2$ as follows:
$$f_1^\pm(x_2,x_3) = \left[ \pm e^{\pm i \epsilon \sqrt{1-x_2^2-x_3^2}}\sqrt{1-x_2^2 - x_3^2} : e^{i\epsilon x_2} x_2 : e^{i\epsilon x_3} x_3 \right],$$
$$f_2^\pm(x_1,x_3) = \left[ e^{i\epsilon x_1} x_1 : \pm e^{\pm i \epsilon \sqrt{1-x_1^2-x_3^2}}\sqrt{1-x_1^2 - x_3^2} : e^{i\epsilon x_3} x_3 \right],$$
$$f_3^\pm(x_1,x_2) = \left[ e^{i\epsilon x_1} x_1 : e^{i\epsilon x_2} x_2 : \pm e^{\pm i \epsilon \sqrt{1-x_1^2-x_2^2}}\sqrt{1-x_1^2 - x_2^2}\right].$$
Observe that in each chart the image of the positive $j^{th}$ hemisphere and the negative $j^{th}$ hemisphere intersect only at the origin (i.e. the poles), and that the perturbations in each chart are consistent with the transition maps.  Thus there are precisely 3 pairs of short Reeb chords for this perturbed trivial cone.  

In summary, we have constructed a family of Lagrangina cones, all isotopic to the trivial cone.  However, for $\epsilon>0$ our cones have the additional property that the projection to $\CC P^{n-1}$ has only 3 transverse double points, while the trivial cone (obtained by taking $\epsilon=0$) is a 2-to-1 cover of it's projection to $\CC P^{n-1}$.  

\end{example}

\section{Legendrian Submanifolds of $S^{2n-1}$ as lifts of Lagrangian Submanifolds in $\CC P^{n-1}$}
\label{sec:LegendrianResults}

The motivation of this paper has been provided by the study of Lagrangian cones.  However, in each case the Lagrangian cones are produced by first lifting an immersion into $\CC P^{n-1}$ to an embedded Legendrian submanifold of $S^{2n-1}$. However, Theorem~\ref{thm:A} and the Main Theorem provide a way to study Legendrian submanifolds of $S^{2n-1}$ on their own. 
 
A great deal of work has been done to study Legendrian knots in dimension $3$, especially in the standard contact $\RR^3$ (cf. \cite{ENSComplete}, \cite{NgThurston}, \cite{legend}, \cite{NgContactHom}, \cite{NgContactHom2}, \cite{NgContactHom3}), and Joshua Sabloff studied the Legendrian contact homology of knots in $3$-dimensional circle bundles in \cite{Sabloff}.

Less is known about Legendrian submanifolds in higher dimensions, and much of it only in the standard contact $\RR^{2n+1}$ (cf. \cite{rotation}, \cite{BaldMcCar2}, \cite{EESContactHom}, \cite{EESContactHom3}).  In \cite{Pati}, Legendrian submanifolds of circle bundles over orbifolds are considered, and in \cite{Asplund}, the circle bundle $\RR^4 \times S^1$ is considered in depth, and related to the case where $\RR^4 \times S^1$ is identified with the Hopf bundle over a single chart of $\CC P^2$ (a special case of Theorem~\ref{thm:A} in this paper). 

Theorem~\ref{thm:A} allows one to study Legendrian submanifolds of $S^{2n-1}$ just as one might study Legendrian submanifolds of $\RR^{2n} \times S^1$ or even the standard contact $\RR^{2n+1}$.  As seen in Example~\ref{ex:HL}, and Section~\ref{sec:HypercubeCones} and \ref{sec:HypercubeCones}, the lifts function in much the same way as one might lift an exact Lagrangian to a Legendrian knot in the standard contact $\RR^{2n+1}$, or the $1$-jet space of a manifold.  

Although Theorem~\ref{thm:A} makes calculations simple, it fails to capture one of the most basic examples: the Legendrian sphere corresponding to the intersection of the trivial cone with $S^{2n-1}$ (as observed in Example~\ref{ex:trivial}).  The Main Theorem moves the story forward, allowing one to consider immersions into $\CC P^{n-1}$ that do not lie in a single chart.  It shows that the calculations are not much more difficult than they are in the case of Theorem~\ref{thm:A}, as in each chart the calculations are standard, and one need only to track how the lifting parameter, $t$, transitions from one chart to the next.  This leads us to ask the following question:

\begin{question}
Sabloff showed in \cite{Sabloff} how to compute the DGA of Legendrian knots in certain contact circle bundles over surfaces.  In the context of Theorem~\ref{thm:A} or the Main Theorem, is there a similar combinatorial algorithm for computing the Legendrian contact homology in higher dimensional circle bundles?
\end{question}

The answer to the previous question may be no.  However, if such an algorithm can be found, one would expect the structure of a radial Lagrangian hypercube diagram to provide a setting in which such calculations would be simple, and could be automated on a computer.

\section{Minimal and Hamiltonian Submanifolds}
\label{sec:MinHam}

Special Lagrangian submanifolds, introduduced by Harvey and Lawson in \cite{HarveyLawson} have been studied extensively due to their connection with mirror symmetry.  Special Lagrangian cones in $\CC^n$ can be studied via the equations that define in them in $\CC^n$, as minimial Legendrians in $S^{2n-1}$ (the link), or from the perspective of the corresponding minimal Lagrangian submanifold of $\CC P^{n-1}$ (cf. \cite{Iriyeh}, and \cite{Haskins}).  While many examples have been studied, the difficulty in working with the special Lagrangian conditions has led to some weaker conditions being studied in the hope of better understanding special Lagrangians.  In \cite{Oh3}, the notion of Hamiltonian minimal (H-minimal) Lagrangian submanifolds was introduced.  A Lagrangian submanifold in a K\"{a}hler manifold is said to be H-minimal if the volume is stationary under compactly supported smooth Hamiltonian deformations (cf. \cite{Iriyeh}).

H-minimal Lagrangian cones in $\CC^2$ were studied and classified by Schoen and Wolfson in \cite{SchoenWolfson1}.   In particular they showed that only cones of Maslov index $\pm 1$ are area minimizing.  Moreover, they showed that if an immersed Lagrangian submanifold of a K\"{a}hler-Einstein manifold is stationary for volume, it is automatically minimal, and special Lagrangian in the Calabi-Yau case (cf. Lemma 8.2 of \cite{SchoenWolfson1}).

It's already known that the trivial cone is H-minimal (cf. \cite{MaOhnita}, \cite{Oh1}, and \cite{Oh2}).  The Harvey-Lawson cone is also known to be strictly Hamiltonian stable, that is, the second variation of the volume is nonnegative under every Hamiltonian deformation, (cf. \cite{Chang} and \cite{MaOhnita}), and it's known that any Hamiltonian stable, minimal Lagrangian torus in $\CC P^2$ is congruent to the Clifford torus (cf. \cite{Oh3}, \cite{Ono}, and \cite{Urbano1}).    

\begin{question}
What are the conditions on a Lagrangian immersion into $\CC P^{n-1}$ that guarantee it lifts to an $H$-minimal Lagrangian cone?
\end{question}


\begin{thebibliography}{99}


\bibitem{NgSYZ} M. Aganagic, T. Ekholm, L. Ng, C. Vafa.  Topological Strings, D-model, and Knot Contact Homology. arXiv:1304.5778 (2013).

\bibitem{Asplund} J. Asplund.  Contact homology of Legendrian knots in five-dimensional circle bundles.  U.U.D.M. Project Report 2016:28

\bibitem{Bald} S. Baldridge.  Embedded and Lagrangian Tori in $\mathbb{R}^4$ and Hypercube Homology. arXiv:1010.3742.

\bibitem{Brunn} H. Brunn, Uber verknotete Kurven. \emph{Verhandlungen de Internationalen Math. Kongresses} (Zurich 1897), Pages 256-259, 1898.

\bibitem{ScottAdam} S. Baldridge, A. Lowrance.  Cube diagrams and 3-dimensional Reidemeister-like Moves for Knots.  \emph{Journal of Knot Theory and its Ramifications}. DOI No:  10.1142/S0218216511009832.

\bibitem{BaldMcCar2} S. Baldridge, B. McCarty.  On the rotation class of knotted Legendrian tori in $\RR^5$.  Topology and its Applications.  Vol 209. 91-114.

\bibitem{Borelli} V. Borrelli, and C. Gorodski.  Minimal Legendrian submanifolds of $S^{2n+1}$ and absolutely area-minimizing cones.  {\em Differential Geometry and its Applications}.  11/2004, pp. 337-347.

\bibitem{Bryant}  R. Bryant.  Some Examples of Special Lagrangian Tori.  \emph{Advances in Theoretical and Mathematical Physics}, 3(1), March 1999.

\bibitem{CastroUrbano} I. Castro, F. Urbano.  New examples of minimal Lagrangian tori in the complex projective plane.  {\em Manuscipta Math}. {\bf 85} (1994) 265-281.

\bibitem{CastroUrbano2} I. Castro, F. Urbano.  On a minimal Lagrangian submanifold of $\CC^n$ foliated by spheres.  {\em Michigan Mathematical Journal}. {\bf 46} (1999) 71-82.

\bibitem{Chang} S. Chang.  On Hamiltonian Stable Minimal Lagrangian Surfaces in $\CC P^2$. \emph{The Journal of Geometric Analysis}.  Vol 10, No. 2, 2000.

\bibitem{CEKSW} G. Civan, J. Etnyre, P. Koprowski, J. Sabloff, A. Walker.  Product structures for Legendrian contact homology.  \emph{Mathematical Proceedings of the Cambridge Philosophical Society}. Vol. 150, Issue 2. 2011.  pp. 291-311.

\bibitem{Cromwell} P. Cromwell. Embedding knots and links in an open book. I. Basic properties. {\em Topology Appl.}, 64 (1995), no. 1, pp. 37-58.

\bibitem{rotation} T. Ekholm, J. Etnyre, and M. Sullivan. Non-isotopic Legendrian submanifolds in $\RR^{2n+1}$.  \emph{Journal of Differential Geometry.} Volume 71, Number 1 (2005), 85-128.

\bibitem{EESContactHom} T. Ekholm, J. Etnyre, and M. Sullivan. The contact homology of Legendrian submanifolds in $\RR^{2n+1}$. \emph{J. Differential Geom.}  Volume 71, Number 2 (2005), 177-305.

\bibitem{EESContactHom2} T. Ekholm, J. Etnyre, and M. Sullivan. Orientations in Legendrian Contact Homology and Exact Lagrangian Immersions.  \emph{International Journal of Mathematics}. 16, 453 (2005). 

\bibitem{EESContactHom3} T. Ekholm, J. Etnyre, and M. Sullivan. Legendrian contact homology in $P\times \RR$.  \emph{Transactions of the American Mathematical Society} 359(7) � June 2005.

\bibitem{EkholmLekili} T. Ekholm, Y. Lekili, Duality between Lagrangian and Legendrian invariants. arXiv:1701.01284.

\bibitem{ENSComplete} T. Ekholm, L. Ng, V. Shende. A complete knot invariant from contact homology.  arXiv:1606.07050.

\bibitem{Eliashberg} Y. Eliashberg. Invariants in contact topology. \emph{Proceedings of the International Congress of Mathematicians}. Vol. II (Berlin, 1998), Doc. Math. 1998, Extra Vol. II, 327?338.

\bibitem{GriffithsHarris} P. Griffiths, J. Harris.  Prinicples of Algebraic Geometry.  Wiley, 1978.

\bibitem{HarveyLawson} R. Harvey, and B. Lawson.  Calibrated geometries. {\em Acta Mathematica}.  Volume 148, No. 1, 1982, 47-157.

\bibitem{Haskins} M. Haskins. Special Lagrangian cones. {\em American Journal of Mathematics}.  Volume 126, No. 4, August 2004, 845-871.

\bibitem{Hitchin} N. Hitchin.  The moduli space of special Lagrangian submanifolds. {\em Annali della Scuola normale superiore di Pisa - Classe di Scienze} Se\'r. 4, 25 No. 3-4 (1997), pp. 503-515.

\bibitem{Iriyeh} H. Iriyeh.  Hamiltonian Minimal Lagrangian Cones in $\CC^m$.  \emph{Tokyo J. Math.} Vol 28, No. 1, 2005.

\bibitem{Joyce3} D. Joyce.  Special Lagrangian 3-folds and integrable systems. {\em Surveys on Geometry and Integrable Systems}. Advanced Studies in Pure Mathematics 51, Mathematical Society of Japan, 2008, pages 189-233.
 
\bibitem{Joyce} D. Joyce. Special Lagrangian $m$-folds in $\CC^m$ with symmetries. {\em Duke Mathematical Journal}.  Volume 115, No. 1 (1992).

\bibitem{Joyce2} D. Joyce.  Special Lagrangian Submanifolds with Isolated Conical Singularities. V. Survey and Application. {\em J. Differential Geom.} Vol. 63. No. 2 (2003), 279-347.

\bibitem{MaOhnita} H. Ma and Y. Ohnita.  Differential geometry of Lagrangian submanifolds and hamiltonian variational problems.  Jan. 2011.  DOI: 10.1090/conm/542/10702

\bibitem{Peter} P. Lambert-Cole.  Legendrian Products.  arXiv:1301.3700.

\bibitem{Peter2} P. Lambert-Cole.  Invariants of Legendrian Products.  Ph.D. Thesis.  http://etd.lsu.edu/docs/available/etd-07142014-122759/.

\bibitem{Ng} L. Ng.  Computable Legendrian Invariants. \emph{Topology}. Vol. 42, Issue 1, January 2003, pp. 55-82.

\bibitem{NgContactHom} L. Ng. Knot and braid invariants from contact homology I. \emph{Geometry and Topology}, 9:247?297, 2005.

\bibitem{NgContactHom2} L. Ng. Knot and braid invariants from contact homology II. \emph{Geometry and Topology}, 9:1603?1637, 2005.

\bibitem{NgContactHom3} L. Ng. Framed knot contact homology. \emph{Duke Mathematics Journal}, 141(2):365?406, 2008.

\bibitem{NgThurston} L. Ng, D. Thurston.  Grid Diagrams, Braids, and Contact Geometry.  \emph{Proceedings of the 13th Gok\"{o}va Geometric-Topology Conference}. pp. 1-17. (2008).

\bibitem{Oh1}  Y. G. Oh, Second variation and stabilities of minimal Lagrangian submanifolds in K\"{a}hler manifolds, \emph{Invent. math.} 101 (1990), 501?519.

\bibitem{Oh2} Y. G. Oh, Tight Lagrangian submanifolds in $\CC P^n$, \emph{Math. Z.} 207 (1991), 409-416.

\bibitem{Oh3} Y. G. Oh, Volume minimization of Lagrangian submanifolds under Hamiltonian deformations, \emph{Math. Z.} 212 (1993), 175?192.

\bibitem{Ono} H. Ono, Hamiltonian stability of Lagrangian tori in toric K\"{a}hler manifolds, \emph{Ann. Glob. Anal. Geom.} 37 (2007), 329?343.

\bibitem{legend} P. Ozvath, Z. Szabo, D. Thurston.  Legendrian knots, transverse knots and combinatorial Floer homology.    arXiv:math/0611841v2.

\bibitem{Pati} J. Pati.  Contact homology of $S^1$-bundles over some symplectically reduced orbifolds.  	arXiv:0910.5934.

\bibitem{Sabloff}  J. Sabloff, Invariants of Legendrian Knots in Circle Bundles.  \emph{Commun. Contemp. Math.}, 05, 569 (2003). 

\bibitem{SchoenWolfson1} R. Schoen and J. Wolfson, Minimizing area among Lagrangian surfaces:  The mapping problem.  \emph{J. Differential Geometry}, {\bf58} (2001) 1-86.

\bibitem{SYZ} A. Strominger, S.T. Yau, and E. Zaslow.  Mirror symmetry and T-duality. {\em Nuclear Physics B}. {\bf 479} (1996) 243-259.

 \bibitem{Urbano1} F. Urbano, Index of Lagrangian submanifolds of $\CC P^n$ and the Laplacian of 1-forms, \emph{Geometria Dedicata} 48 (1993), 309-318.
 
\bibitem{Wolfson} J. Wolfson, Two applications of prequantization in Lagrangian topology, Pacific Journal of Math.  215 (2004) 393-398.

\end{thebibliography}
\end{document}